\documentclass[10pt]{article}
\usepackage[utf8]{inputenc}
\usepackage[a4paper,margin=1.2in]{geometry}

\title{Global unique solutions with instantaneous loss of regularity for SQG with fractional diffusion}
\author{Diego C\'ordoba\footnote{dcg@icmat.es}\quad and Luis Mart\'inez-Zoroa\footnote{luis.martinez@icmat.es}\\ \\ \small Instituto de Ciencias Matem\'aticas CSIC-UAM-UCM-UC3M }

\usepackage{graphicx}
\usepackage{amsmath}
\usepackage{ dsfont }
\usepackage{amsfonts}
\usepackage{amsthm}
\usepackage{ amssymb }
\usepackage{xcolor}

\newtheorem{theorem}{Theorem}[section]
\newtheorem{corollary}{Corollary}[theorem]
\newtheorem{lemma}[theorem]{Lemma}
\newtheorem{remark}{Remark}
\def\p{\partial}

\begin{document}
\maketitle

\begin{abstract}
    In this work we construct global unique solutions of the dissipative Surface quasi-geostrophic equation ($\alpha$-SQG) that lose regularity instantly when there is  super-critical fractional diffusion.
    
\end{abstract}

\section{Introduction}
The inviscid Surface Quasi-Geostrophic (SQG) equation is a significant active scalar model with various applications in atmospheric modeling \cite{P}, owing to its similarities with the 3D incompressible Euler equations (see \cite{CMT}). In this work we consider the initial value problem for the dissipative 2D  Surface Quasi-geostrophic equations ($\alpha$-SQG)
in the space-time domain $\mathds{R}^2 \times \mathds{R}_+$ which has the following form
\begin{eqnarray}\label{SQG}
\frac{\partial w}{\partial t} + v\cdot\nabla w + \Lambda^{\alpha} w  = 0\quad\quad \alpha\in(0,2]\\
v = (- \frac{\partial \psi}{\partial x_2}, \frac{\partial \psi}{\partial x_1}), \quad \psi = \Lambda^{-1} w \nonumber
\end{eqnarray}
where we denote $\Lambda^{\alpha} f\equiv (-\Delta)^{\frac{\alpha}{2}} f$ by the Fourier transform $\widehat{\Lambda^{\alpha} f}(\xi) = |\xi|^{\alpha}\widehat{f} (\xi)$.   Here the function $w= w(x,t)$ represents the potential temperature in a rapidly rotating and stratified flow driven by an incompressible velocity $v$. The velocity field can be written as $v=(-\mathcal{R}_{2} w, \mathcal{R}_{1} w)$,
where $\mathcal{R}_{i}$ are the Riesz transforms in 2 dimensions, with the integral expression
$$\mathcal{R}_{j} w(x,t)=\frac{\Gamma (3/2)}{\pi^{3/2}}P.V. \int_{\mathds{R}^2} \frac{(x_{j}-y_{j})w(y,t)}{|x-y|^{3}}dy_{1}dy_{2} $$
for $j=1,2.$


The equation (\ref{SQG}) has been extensively studied since its introduction in \cite{R}. In that work, the global existence of weak solutions in $L^2$ (finite energy) was demonstrated for $0 \leq \alpha \leq 2$. Further research on global existence of weak solutions in other spaces can be found in \cite{Ma} and \cite{BG}. However, it should be noted that weak solutions are not unique below a certain regularity threshold \cite{BSV}.

The equation's scaling leads to three regimes to consider: sub-critical ($1<\alpha\leq 2$), critical ($\alpha=1$), and super-critical ($0<\alpha<1$). The global existence of unique smooth solutions in the sub-critical case has been established in \cite{CW1}, while the global well-posedness for the critical case with $\alpha=1$    has been shown in \cite{KNV},  \cite{CV} and \cite{CVi}  using different techniques  (see also \cite{KN}, \cite{DKSV} and \cite{CTV}).



\subsection{Regularity in the Super-critical regime $\alpha\in(0,1)$}

The problem of global regularity in the supercritical regime remains unresolved, despite the existence of  eventual regularity results in \cite{D}, \cite{CZV}, \cite{K}, \cite{LX}, \cite{SGS} and \cite{Si}. The local well-posedness  has been established  for large data in $H^s$ for $s\geq 2-\alpha$ (see \cite{M}) and for a number of functional spaces global well-posedness is present for small data (see \cite{CC}, \cite{ChL}, \cite{ChMZ}, \cite{CW2} \cite{Do}, \cite{DoLi}, \cite{HK}, \cite{Ju}, \cite{M}, \cite{SGS}, \cite{W1}, \cite{W2}, \cite{W3} and \cite{XZ}). In the case of large initial data global existence as $\alpha\rightarrow 1^-$  is shown in \cite{CZV}  (see also \cite{CH}). Additionally, there is a corresponding instant parabolic smoothing effect for sufficiently regular initial data \cite{ConstantinWu}, \cite{CW2},  \cite{Bi}, \cite{Do} and \cite{DoLi}.

Recently in \cite{CH}, a bound is obtained on  the dimension of the spacetime singular set of the suitable weak solutions of (\ref{SQG}) for a range of $\alpha$'s in the super-critical regime.

\subsection{Main result}
Our main result is to construct global unique solutions of (\ref{SQG}) that loses regularity instantly in the  super-critical regime.
\begin{theorem}
Given $\epsilon>0$, $\alpha\in(0,1)$, $\beta\in (1,2-\alpha)$, there exist initial conditions $w_{0}(x)$ with $||w_{0}||_{H^{\beta}}\leq \epsilon$ such that there exists a unique solution $w(x,t)$ to \eqref{SQG} with $w(x,t)\in L_{t}^{\infty}H^{1}_{x}$. This solution is global and smooth for any $t>0$ and it fulfils

    $$\text{lim}_{n\rightarrow\infty}||w(x,t_{n})||_{H^{\beta}}=\infty$$
    for some sequence of times $(t_{n})_{n\in \mathds{N}}$ that tends to zero.
\end{theorem}

\begin{remark}
One expects that, as $\alpha$ becomes bigger, instant loss of regularity should become harder and, in fact, if we consider $L^{\infty}$ initial conditions the result obtained in \cite{LX} (see also \cite{L})    shows that, in the critical case $\alpha=1$, for $s\in(0,1)$, there exists at least one local weak solution that does not lose regularity, which suggests that there might not be instant loss of regularity for $L^{\infty}$ functions in the case $\alpha\geq 1$. This is also supported by the global existence results for $\alpha=1$ \cite{CV}, \cite{KNV} and \cite{CVi}.  
\end{remark}

\begin{remark}
The growth around the origin is at least logarithmic, i.e., there is an exponent $a>0$ such that

$$\text{lim}_{n\rightarrow\infty}\frac{||w(x,t_{n})||_{H^{\beta}}}{|\ln(t_{n})|^{a}}>0.$$
 We will however omit the proof of this fact in order to obtain a more readable paper.

\end{remark}

\begin{remark}
    The solution $w(x,t)$ converges to the initial conditions $w(x,0)$ in the space $C^{3}_{x}(B_{J}(0))$ for any $J$ as $t$ tends to $0$. This is a trivial consequence of \eqref{convc3trunc}.
\end{remark}

\subsection{Strategy of the proof}

The motivation for studying this problem comes from the work \cite{Zoroacordoba} where the authors show instant loss of regularity in the inviscid case (see also \cite{Injee} for a different proof). This is done in two steps. First we construct a pseudo-solution $\bar{w}= \bar{w}(x,t)$ to SQG i.e. a solution to the following equation
$$\frac{\partial \bar{w}}{\partial t} + v_{1}(\bar{w})\frac{\partial \bar{w}}{\partial x_{1}}+ v_{2}(\bar{w})\frac{\partial \bar{w}}{\partial x_{2}}+F(x,t)=0$$
$$v_{1}(\bar{w})=-\frac{\partial}{\partial x_{2}}(-\Delta)^{-1/2} \bar{w}=-\mathcal{R}_{2}\bar{w}$$
$$v_{2}(\bar{w})=\frac{\partial}{\partial x_{1}}(-\Delta)^{-1/2} \bar{w}=\mathcal{R}_{1}\bar{w}$$
that exhibits norm inflation, with $F(x,t)$ small and  with enough regularity.  This allows us to show that the unique classical solution $w=w(x,t)$ to SQG with the same initial data, $\bar{w}(x,0)=w(x,0),$ also exhibits norm inflation. 

The second step consists in using a gluing argument to combine an infinite number of these rapidly growing solutions to obtain loss of regularity. This result suggest that at least for small diffusion we expect to prove a similar result for the equation $\alpha$-SQG (\ref{SQG}).

In order to obtain norm inflation in $H^s$ with diffusion we consider similar (but more general) initial conditions in polar coordinates $(r,\theta)$ as in \cite{Zoroacordoba}, namely
$$w(r,\theta,0)\equiv w_{rad}(r,0)+w_{pert}(r,\theta,0)= \frac{f(\lambda r)}{\lambda^{\beta-1}}+g(\lambda r)\frac{\cos(N\theta)}{N^{\beta}\lambda^{s-1}}$$
where $N\in \mathds{N},\lambda\in\mathds{R}$ are parameters to be fixed later and $f,g$ are smooth radial functions. 

To obtain  a reliable pseudo-solution,  we aim to find an approximation of $\alpha$-SQG that is simple enough to be solved explicitly, and the pseudo-solution will be obtained by solving the simplified evolution equation.  This simplified evolution equation needs to be precise enough that the pseudo-solutions stay close to the actual solutions to $\alpha$-SQG. When fractional dissipation is absent and $\lambda$ is sufficiently large, we can rely on the following equations

$$\frac{\partial}{\partial t}w_{rad}=0$$
$$\frac{\partial}{\partial t}w_{pert}+v(w_{rad})\cdot\nabla w_{pert}=0$$
to obtain a pseudo-solution that grows rapidly in time, and taking $N$ big makes this pseudo-solution a good approximation of SQG. Unfortunately, this ansatz for the evolution  would completely ignore the diffusion in the case $\alpha>0$, which would make the pseudo-solution a very poor approximation of $\alpha$-SQG. On the other hand, including the fractional diffusion in our simplified evolution equation already produces an equation that is too complicated for our purposes, and in particular it is hard to deal with the non-locality of $\Lambda^{\alpha}$. Before we explain in some detail how to circumvent this, we will study what we call the naive pseudo-solution:

$$\bar{w}_{naive}(r,\theta,t)=e^{-C\lambda^{\alpha}t}\frac{g(\lambda r)}{\lambda^{s-1}}+e^{-C(N\lambda)^{\alpha}t}f(\lambda r)\frac{\cos(N(\theta-\frac{v_{\theta}g(\lambda r)}{r}\int_{0}^{t}e^{-C\lambda^{\alpha}s}ds))}{N^{s}\lambda^{s-1}}$$
with $v_{\theta}(g(\lambda r))=\hat{\theta}\cdot v(g(\lambda r))$, which is obtained by using the (fully local) approximation 

$$\Lambda^{\alpha}w_{rad}\approx C\lambda^{\alpha}w_{rad},$$
$$\Lambda^{\alpha}w_{pert}\approx C(\lambda N)^{\alpha} w_{pert}.$$

This ansatz (which is NOT a good approximation of $\alpha$-SQG) actually gives some basic ideas of what the behaviour for the real solution is going to be. Namely, we see that the characteristic time for the decay of $w_{pert}$ is $(\lambda N)^{-\alpha}$, while the "deformation time" (i.e., the time it would take for $||w_{pert}||_{H^\beta}$ to grow in the absence of diffusion) is of order $\lambda^{-2+\beta}$, which suggests considering $(\lambda N)^{-\alpha}\approx \lambda^{-2+\beta}$, so that the smoothing effects and the deformation effects have roughly the same strength. Note, in particular, that this already suggests that we can only have instantaneous loss of regularity if $\beta< 2-\alpha$, which is consistent with the fact that there is local well posedness in $H^{\beta}$ for $\beta\geq 2-\alpha.$

However $w_{naive}$ is not the right approximation, so to actually include the diffusion in our simplified evolution equation, we will compute $\bar{\Lambda}^{\alpha}$, a local approximation of the diffusion, as well as $\bar{v}$, a local approximation of the velocity operator , to obtain the final version of our simplified equations

$$\frac{\partial}{\partial t} \bar{w}_{rad}(r,t)+\Lambda^{\alpha}\bar{w}_{rad}(r,t)=0$$

$$\frac{\partial \bar{w}_{pert}}{\partial t}+v(\bar{w}_{rad}(r,t))\cdot\nabla(\bar{w}_{pert})+\bar{v}(\bar{w}_{pert})\cdot \nabla \bar{w}_{rad}(r,t)+\bar{\Lambda}^{\alpha}(\bar{w}_{pert})=0,$$
which can be solved explicitly. This pseudo-solution, which depends on $\lambda$, $N$, $\beta$ and $\alpha$, grows very rapidly in $H^{\beta}$ as long as $N$ and $\lambda$ are chosen correctly and $\beta<2-\alpha$. Furthermore, if $\beta>1$ (and again, $N$ and $\lambda$ chosen correctly), the pseudo-solution is a good approximation of $\alpha$-SQG for all time $t>0$. A gluing argument then allows us to combine an infinite number of these rapidly growing solutions to obtain the desired instant loss of regularity.

The approach of using pseudo-solutions to obtain information about norm inflation and loss of regularity has been used for inviscid SQG \cite{Zoroacordoba}, the 2D incompressible Euler equations \cite{CMZO} and for generalized SQG \cite{CMZ2}, but this is the first work where an equation with dissipation is considered.

\subsection{Outline of the paper}
This paper is structured as follows. In Section 2, we introduce the basic notation that will be utilized throughout the paper and derive necessary technical bounds to approximate the diffusion operator and to find and control the pseudo-solution. In Section 3, we present the pseudo-solution and analyze its essential properties. In Section 4, we demonstrate how a gluing argument can be employed to construct a global unique in time solution, despite a loss of regularity.

\section{Technical lemmas}

\subsection{Notation and preliminaries}

When a constant depends on several parameters (such as $\alpha$, $\beta$, and $\gamma$), we will use the notation $C_{\alpha,\beta,\gamma}$ to indicate this dependence in this paper.
We will, however, omit the sub-index if the parameter has been fixed at the time.

For many lemmas it will be necessary to work in polar coordinates, i.e., we will consider the change of coordinates

$$x_{1}=r\cos(\theta), x_{2}=r\sin(\theta).$$

Furthermore, if we call $F_{polar}$ the function that gives us the change of coordinates from polar to cartesian coordinates, for some function $f(x)$ we will use the abuse of notation

$$f(r,\alpha)=f(F_{polar}(r,\alpha)).$$

For $s\geq 0$, we will consider the $H^{s}$ norms, which we will define as

$$||f||_{H^{s}}=||f||_{L^2}+||\Lambda^{s}f||_{L^2},$$
and sometimes we will use the fact that, for $s$ an integer

$$||f||_{H^{s}}\approx \sum_{j=0}^{s}\sum_{i=0}^{j}||\frac{\partial^{j}}{\partial x_{1}^{i}\partial x_{2}^{j-i}}f||_{L^2}.$$

Finally, we will sometimes consider the homogeneous Sobolev norms, defined as

$$||f||_{\dot{H}^{s}}=||\Lambda^{s}f||_{L^2}.$$

\subsection{Approximations for the fractional diffusion}

As we mentioned in the introduction, in order to obtain  an appropriate pseudo-solution, we need an approximation for the diffusion operator that is easier to work with, in particular we would like a local approximation for the operator. Doing this directly for $\Lambda^{\alpha}$ posses some difficulties due to the lack of integrability of the kernel of $\Lambda^{\alpha}$, so we will first approximate $\Lambda^{-\alpha}$ and then use that to obtain information about $\Lambda^{\alpha}$.

\begin{lemma}\label{-alpha}
For any fixed parameters $\alpha\in(0,1]$, $P,\epsilon>0$ there exists $N_{0}$ such that if $N>N_{0}$, then for any functions $f(r)$, $g(r)$ and $p(r)$ fulfilling $\text{supp}f(r)\subset(\frac12,\frac32)$ and

$$||g||_{C^5}\leq (\ln{(N)})^{P}, \ ||f(r)||_{C^{5}}\leq  (\ln{(N)})^{P} ||f||_{L^{\infty}},||p(r)||_{C^{5}}\leq (\ln{(N)})^{P}$$
then if we define
$$w(r,\theta):=f(r)\cos(N(\theta+g(r))+p(r))$$
we have that for $\beta\in[0,3]$ there exist constants $K_{\alpha}>0$ and $C_{\epsilon,\alpha,P}$ such that 
$$||\Lambda^{-\alpha}w(r,\theta)-K_{\alpha}\frac{w(r,\theta)}{|(\frac{N}{r})^2+(N)^2g'(r)^2|^{\alpha/2}}||_{H^{\beta}}\leq C_{\epsilon,\alpha,P} N^{-1-\alpha+\epsilon+\beta}||f||_{L^{\infty}}.$$

Furthermore if we have $f(r),g(r),p(r)$ with $\text{supp} f(r)\subset(\frac{1}{2\lambda},\frac{3}{2\lambda})$ for some $\lambda\geq 1$ and such that we have

$$||g(\frac{r}{\lambda})||_{C^5}\leq (\ln{(N)})^{P}, \ ||f(\frac{r}{\lambda})||_{C^{5}}\leq  (\ln{(N)})^{P} ||f(\frac{r}{\lambda})||_{L^{\infty}},||p(\frac{r}{\lambda})||_{C^{5}}\leq  (\ln{(N)})^{P}$$
then for $\beta\in[0,3]$ there exist constants $K_{\alpha}>0$ and $C_{\epsilon,\alpha,P}$ such that 
$$||\Lambda^{-\alpha}w(r,\theta)-K_{\alpha}\frac{w(r,\theta)}{|(\frac{N}{r})^2+(N)^2g'(r)^2|^{\alpha/2}}||_{H^{\beta}}\leq C_{\epsilon,\alpha,P,M} \lambda^{\beta-\alpha-1} N^{-1-\alpha+\epsilon+\beta}||f||_{L^{\infty}}.$$

\end{lemma}

\begin{proof}

We will just consider $P=1$ for simplicity since the proof is the same for other values of $P$.
We start by proving the result for $\beta=0$ and $\lambda=1$.  We will consider from now on that $\alpha$ is fixed with $\alpha\in(0,1)$, so we will omit the dependence of the constants with respect to $\alpha$. First, we have that, in polar coordinates

\begin{align*}
    &\Lambda^{-\alpha}w(r,\theta)=\int_{-\pi}^{\pi}\int_{0}^{\infty} \frac{w(r',\theta')}{|(r-r')^2+2rr'(1-\cos(\theta-\theta'))|^{\frac{2-\alpha}{2}}}r'dr'd\theta'\\
    &=\int_{-\pi}^{\pi}\int_{-r}^{\infty} \frac{f(r+h)\cos(N\tilde{\theta})\cos(N\theta+Ng(r+h)+p(r+h))}{|h^2+2r(r+h)(1-\cos(\tilde{\theta}))|^{\frac{2-\alpha}{2}}}(r+h)dhd\tilde{\theta}.
\end{align*}

Since

\begin{align*}
    &||\Lambda^{-\alpha}w(r,\theta)-K_{\alpha}\frac{w(r,\theta)}{|(\frac{N}{r})^2+N^2g'(r)^2|^{\alpha/2}}||_{L^{2}}\\
    &\leq||(\Lambda^{-\alpha}w(r,\theta)-K_{\alpha}\frac{w(r,\theta)}{|(\frac{N}{r})^2+N^2g'(r)^2|^{\alpha/2}})1_{r\in(\frac14,2)}||_{L^{2}}\\
    &+ ||(\Lambda^{-\alpha}w(r,\theta)-K_{\alpha}\frac{w(r,\theta)}{|(\frac{N}{r})^2+N^2g'(r)^2|^{\alpha/2}})1_{r\notin(\frac14,2)}||_{L^{2}}=I_{1}+I_{2}
\end{align*}
we  will start by studying the operator when $r\in(\frac14,2)$, so that we can bound $I_{1}$. 

First, we note that
$$\frac{\partial^{k}}{\partial \tilde{\theta}^{k}}\frac{1}{|h^2+2r(r+h)(1-\cos(\tilde{\theta}))|^{\frac{2-\alpha}{2}}}=\sum_{i=1}^{k}\frac{((r+h)r)^{i}P_{k,i}(\tilde{\theta})}{|h^2+2r(r+h)(1-\cos(\tilde{\theta}))|^{\frac{2+2i-\alpha}{2}}}$$
with 
\begin{align}\label{Pki}
    P_{k,i}(\tilde{\theta})=\sum_{(j,l)\in S_{k,i}}c_{k,i,j,l}(\cos(\tilde{\theta}))^{j}(\sin(\tilde{\theta}))^{l},
\end{align}
\begin{align}\label{Ski}
    S_{k,i}:=\{(j,l)\in(0\cup\mathds{\mathds{N}})^2:l\geq i-(k-i),j+l=i\}.
\end{align}
To prove this, we note that the equality is trivially true for the case $k=0$ and
\begin{align*}
    &\frac{\partial}{\partial \tilde{\theta}}\frac{((r+h)r)^{i}(\cos(\tilde{\theta}))^{j}(\sin(\tilde{\theta}))^{l}}{(h^2+2r(r+h)(1-\cos(\tilde{\theta})))^{\frac{2+2i-\alpha}{2}}}\\
    &=-(2+2i-\alpha)\frac{((r+h)r)^{i+1}(\cos(\tilde{\theta}))^{j}(\sin(\tilde{\theta}))^{l+1}}{(h^2+2r(r+h)(1-\cos(\tilde{\theta})))^{\frac{2+2(i+1)-\alpha}{2}}}\\
    &-j\frac{((r+h)r)^{i}(\cos(\tilde{\theta}))^{j-1}(\sin(\tilde{\theta}))^{l+1}}{(h^2+2r(r+h)(1-\cos(\tilde{\theta})))^{\frac{2+2i-\alpha}{2}}}+l\frac{((r+h)r)^{i}(\cos(\tilde{\theta}))^{j+1}(\sin(\tilde{\theta}))^{l-1}}{|h^2+2r(r+h)(1-\cos(\tilde{\theta}))|^{\frac{2+2i-\alpha}{2}}}.
\end{align*}
This means that, if $C_{k,i,j,l}$ is  different from zero, by taking a derivative we get contributions to $c_{k+1,i+1,j,l+1},\ c_{k+1,i,j-1,l+1}$ and $c_{k+1,i,j+1,l-1}$. But this coefficients also fulfil the inequalities in \eqref{Ski}, which finishes the proof.

With this in mind, using integration by parts with respect to $\tilde{\theta}$  $k$ times,
\begin{align}\label{porpartes}
    &\int_{-\pi}^{\pi} \frac{\cos(N\tilde{\theta})}{|h^2+2r(r+h)(1-\cos(\tilde{\theta}))|^{\frac{2-\alpha}{2}}}d\tilde{\theta}\\\nonumber
    &=\int_{-\pi}^{\pi}\frac{\cos(N\tilde{\theta}+k\frac{\pi}{2})}{N^{k}}\sum_{i=1}^{k}\frac{((r+h)r)^{i}P_{k,i}(\tilde{\theta})}{|h^2+2r(r+h)(1-\cos(\tilde{\theta}))|^{\frac{2+2i-\alpha}{2}}}d\tilde{\theta}\\\nonumber
\end{align}
with $P_{k,i}$ as given in \eqref{Pki}.

This implies that, for $(r+h)\in \text{supp}(f)$

$$\sum_{i=0}^{k}\frac{((r+h)r)^{i}P_{k,i}(\tilde{\theta})}{|h^2+2r(r+h)(1-\cos(\tilde{\theta}))|^{\frac{2+2i-\alpha}{2}}}\leq \frac{C_{k}}{|h^2+2r(r+h)(1-\cos(\tilde{\theta}))|^{\frac{2+k-\alpha}{2}}}.$$

Therefore we get, for any $\epsilon'>0$

\begin{align*}
    &|\int_{-\pi}^{\pi}\int_{[-r,\infty ]\setminus[-N^{-1+\epsilon'},N^{-1+\epsilon'}]} \frac{f(r+h)\cos(N\tilde{\theta})\cos(N\theta+Ng(r+h)+p(r+h))}{|h^2+2r(r+h)(1-\cos(\tilde{\theta}))|^{\frac{2-\alpha}{2}}}(r+h)dhd\tilde{\theta}|\\
    &\leq C_{k} \frac{||f||_{L^{\infty}}}{N^{k}}N^{(1-\epsilon')(2+k-\alpha)}
\end{align*}
and this can be bounded by $C_{\epsilon'} N^{-1-\alpha}$ by taking $k$ big enough.



We can therefore focus only on the integral when $(h,\tilde{\theta})\in[-N^{-1+\epsilon'},N^{-1+\epsilon'}]\times[-\pi,\pi]$, and in fact by symmetry it is enough to study $(h,\tilde{\theta})\in[-N^{-1+\epsilon'},N^{-1+\epsilon'}]\times[0,\pi]$. For this set, we will make a couple of approximations for our kernel that will only produce a small error, namely, we note that, using integration by parts with respect to $\tilde{\theta}$

\begin{align*}
    &|\int_{0}^{\pi}\int_{-N^{-1+\epsilon'}}^{N^{-1+\epsilon'}} f(r+h)\cos(N\tilde{\theta})\Big(\frac{\cos(N\theta+Ng(r+h)+p(r+h))}{|h^2+2r(r+h)(1-\cos(\tilde{\theta}))|^{\frac{2-\alpha}{2}}}\\
    &-\frac{\cos(N\theta+Ng(r+h)+p(r+h))}{|h^2+r(r+h)\tilde{\theta}^2|^{\frac{2-\alpha}{2}}}\Big)(r+h)dhd\tilde{\theta}|\\
    &\leq C\frac{||f||_{L^{\infty}}}{N}\int_{-N^{-1+\epsilon'}}^{N^{-1+\epsilon'}}\int_{0}^{\pi} \Big|\frac{2r(r+h)\sin(\tilde{\theta})}{|h^2+2r(r+h)(1-\cos(\tilde{\theta}))|^{\frac{4-\alpha}{2}}}-\frac{2r(r+h)\tilde{\theta}}{|h^2+r(r+h)\tilde{\theta}^2|^{\frac{4-\alpha}{2}}}\Big|d\tilde{\theta}dh\\
    &\leq C\frac{||f||_{L^{\infty}}}{N}\Big(\int_{-N^{-1+\epsilon'}}^{N^{-1+\epsilon'}}\int_{0}^{\pi}\frac{1}{|h^2+\tilde{\theta}^2|^{\frac{1-\alpha}{2}}}d\tilde{\theta}dh\Big)\\
    &\leq C||f||_{L^{\infty}}N^{-2+\epsilon'}
\end{align*}


We also have, for large $N$,
\begin{align*}
    &|\int_{0}^{\pi}\int_{-N^{-1+\epsilon'}}^{N^{-1+\epsilon'}} f(r+h)\Big(\frac{(r+h)\cos(N\tilde{\theta})\cos(N\theta+Ng(r+h)+p(r+h))}{|h^2+r(r+h)\tilde{\theta}^2|^{\frac{2-\alpha}{2}}}\\
    &-\frac{R\cos(N\tilde{\theta})\cos(N\theta+Ng(r+h)+p(r+h))}{|h^2+r^2\tilde{\theta}^2|^{\frac{2-\alpha}{2}}}\Big)dhd\tilde{\theta}|\\
    &\leq ||f||_{L^{\infty}}C\int_{0}^{\pi}\int_{-N^{-1+\epsilon'}}^{N^{-1+\epsilon'}} \bigg(\frac{h\tilde{\theta}^2}{|h^2+r(r+h)\tilde{\theta}^2|^{\frac{4-\alpha}{2}}}+\frac{h}{|h^2+r(r+h)\tilde{\theta}^2|^{\frac{2-\alpha}{2}}}\bigg)dhd\tilde{\theta}\\
    &\leq ||f||_{L^{\infty}}CN^{-1+\epsilon'}\int_{-N^{-1+\epsilon'}}^{N^{-1+\epsilon'}}\frac{1}{h^{1-\alpha}}dh\leq ||f||_{L^{\infty}}CN^{(-1+\epsilon')(1+\alpha)}.
\end{align*}

All these inequalities combined already give us

\begin{align*}
    &|\Lambda^{-\alpha}w(r,\theta)-\int_{-\pi}^{\pi}\int_{-N^{-1+\epsilon'}}^{N^{-1+\epsilon'}} \frac{f(r+h)\cos(N\tilde{\theta})\cos(N\theta+Ng(r+h)+p(r+h))}{|h^2+r^2\tilde{\theta}^2|^{\frac{2-\alpha}{2}}}rdhd\tilde{\theta}|\\
    &\leq C_{\alpha,\epsilon'}||f||_{L^{\infty}}N^{-(1+\alpha)(1-\epsilon')},
\end{align*}
and furthermore, we have that

\begin{align*}
    &|\int_{-\pi}^{\pi}\int_{-N^{-1+\epsilon'}}^{N^{-1+\epsilon'}} \frac{f(r+h)\cos(N\tilde{\theta})(\cos(N\theta+Ng(r+h)+p(r+h))-\cos(N\theta+Ng(r+h)+p(r)))}{|h^2+r^2\tilde{\theta}^2|^{\frac{2-\alpha}{2}}}rdhd\tilde{\theta}|\\
    &\leq C||f||_{L^{\infty}}\ln{(N)}N^{-1+\epsilon'}\int_{0}^{\pi}\int_{0}^{N^{-1+\epsilon'}} \frac{1}{(h+\tilde{\theta})^{2-\alpha}}dhd\tilde{\theta}\leq C||f||_{L^{\infty}}\ln{(N)}N^{-(1+\alpha)(1-\epsilon')},
\end{align*}
and
\begin{align*}
    &|\int_{-\pi}^{\pi}\int_{-N^{-1+\epsilon'}}^{N^{-1+\epsilon'}} \frac{(f(r+h)-f(r))\cos(N\tilde{\theta})\cos(N\theta+Ng(r+h)+p(r))}{|h^2+r^2\tilde{\theta}^2|^{\frac{2-\alpha}{2}}}rdhd\tilde{\theta}|\\
    &\leq C||f||_{L^{\infty}}\ln{(N)}N^{-1+\epsilon'}\int_{0}^{\pi}\int_{0}^{N^{-1+\epsilon'}} \frac{1}{(h+\tilde{\theta})^{2-\alpha}}dhd\tilde{\theta}\leq C||f||_{L^{\infty}}\ln{(N)}N^{-(1+\alpha)(1-\epsilon')},
\end{align*}
and

\begin{align*}
    &|\int_{-\pi}^{\pi}\int_{-N^{-1+\epsilon'}}^{N^{-1+\epsilon'}} f(r))\cos(N\tilde{\theta})\frac{1}{|h^2+r^2\tilde{\theta}^2|^{\frac{2-\alpha}{2}}}r\\
    &(\cos(N\theta+Ng(r+h)+p(r))-\cos(N\theta+Ng(r)+Nhg'(r)+p(r))dhd\tilde{\theta}|\\
    &\leq C\ln{(N)}N^{-1+2\epsilon'}\int_{0}^{\pi}\int_{0}^{N^{-1+\epsilon'}} \frac{1}{(h+\tilde{\theta})^{2-\alpha}}dhd\tilde{\theta}\leq C\ln{(N)}N^{-(1+\alpha)(1-\epsilon')+\epsilon'}.
\end{align*}



Therefore, we just need to study
\begin{align*}
    &\int_{-\pi}^{\pi}\int_{-N^{-1+\epsilon'}}^{N^{-1+\epsilon'}} \frac{f(r)\cos(N\tilde{\theta})\cos(N\theta+Ng(r)+Nhg'(r)+p(r))}{|h^2+r^2\tilde{\theta}^2|^{\frac{2-\alpha}{2}}}rdhd\tilde{\theta}\\
    &=\int_{-r\pi }^{r\pi}\int_{-N^{-1+\epsilon'}}^{N^{-1+\epsilon'}} \frac{f(r)\cos(\frac{N}{r}\bar{\theta})\cos(N\theta+Ng(r)+Nhg'(r)+p(r))}{|h^2+\bar{\theta}^2|^{\frac{2-\alpha}{2}}}dhd\bar{\theta}\\
    &=N^{-\alpha}\int_{-r\pi N}^{r\pi N}\int_{-N^{\epsilon'}}^{N^{\epsilon'}} \frac{f(r)\cos(\frac{1}{r}s_{2})\cos(N\theta+Ng(r)+s_{1}g'(r)+p(r))}{|s_{1}^2+s_{2}^2|^{\frac{2-\alpha}{2}}}ds_{1}ds_{2}\\
    &=N^{-\alpha}\cos(N\theta+Ng(r)+p(r))\int_{-r\pi N}^{r\pi N}\int_{-N^{\epsilon'}}^{N^{\epsilon'}} \frac{f(r)\cos(\frac{1}{r}s_{2})\cos(s_{1}g'(r))}{|s_{1}^2+s_{2}^2|^{\frac{2-\alpha}{2}}}ds_{1}ds_{2},
\end{align*}

With this in mind, we want to show that
\begin{align}\label{H*}
    H^{*}:=\lim_{N\rightarrow} H_{N}&:=\lim_{N\rightarrow\infty}\int_{-r\pi N}^{r\pi N}\int_{-N^{\epsilon'}}^{N^{\epsilon'}} \frac{\cos(\frac{1}{r}s_{2})\cos(s_{1}g'(r))}{|s_{1}^2+s_{2}^2|^{\frac{2-\alpha}{2}}}ds_{1}ds_{2}\nonumber \\
    &=\frac{K_{\alpha}}{\big(\big(\frac{1}{r}\big)^2+g'(r)^2\big)^{\frac{\alpha}{2}}}
\end{align}
with $K_{\alpha}>0$ and also that, for $N_{2}\geq N_{1}$

\begin{equation}\label{HN1N2}
    |H_{N_{1}}-H_{N_{2}}|\leq C_{\epsilon',\alpha}N_{1}^{-1+\epsilon'},
\end{equation}
so in particular $|H^{*}-H_{N}|\leq C_{\epsilon'}N^{-1+\epsilon'}$.
We start by obtaining (\ref{HN1N2}). Note that

$$H_{N_{2}}-H_{N_{1}}=\int_{A_{N_{1},N_{2}}\cup B_{N_{1},N_{2}}} \frac{\cos(\frac{1}{r}s_{2})\cos(s_{1}g'(r))}{|s_{1}^2+s_{2}^2|^{\frac{2-\alpha}{2}}}ds_{1}ds_{2}$$
with 
$$A_{N_{1},N_{2}}:=[N_{1}^{\epsilon'},N_{2}^{\epsilon'}]\times[-r\pi N_{2},r\pi N_{2}]\cup [-N_{2}^{\epsilon'},-N_{1}^{\epsilon'}]\times[-r\pi N_{2},r\pi N_{2}]$$

$$B_{N_{1},N_{2}}:=[-N_{1}^{\epsilon'},N_{1}^{\epsilon'}]\times[-r\pi N_{2},-r\pi N_{1}]\cup [-N_{1}^{\epsilon'},N_{1}^{\epsilon'}]\times[r\pi N_{1},r\pi N_{2}].$$

We now bound the integral over $A_{N_{1},N_{2}}$, we will focus on the part with $s_{1}>0$, the other half being analogous.

Applying integration by parts $k$ times with respect to the variable $s_{2}$  we get

\begin{align*}
  &|\int_{N_{1}^{\epsilon'}}^{N_{2}^{\epsilon'}}\int_{-r\pi N_{2}}^{r\pi N_{2}} \frac{\cos(\frac{1}{r}s_{2})\cos(s_{1}g'(r))}{|s_{1}^2+s_{2}^2|^{\frac{2-\alpha}{2}}}ds_{2}ds_{1}|\\ 
  &\leq \Big(\frac{C_{k} N_{2}^{\epsilon'}}{N_{2}^{2-\alpha}}+C_{k}\int_{N_{1}^{\epsilon'}}^{N_{2}^{\epsilon'}}\int_{-r\pi N_{2}}^{r\pi N_{2}}\frac{1}{|s_{1}^2+s_{2}^2|^{\frac{2-\alpha+k}{2}}}ds_{2}ds_{1}\Big)\\
  &\leq \Big(\frac{C_{k}N_{2}^{\epsilon'}}{N_{2}^{2-\alpha}}+C_{k}\int_{N_{1}^{\epsilon'}}^{N_{2}^{\epsilon'}}\int_{-r\pi N_{2}}^{r\pi N_{2}}\frac{1}{|s_{1}+s_{2}|^{2-\alpha+k}}ds_{2}ds_{1}\Big)\\
  &\leq C_{k}N_{2}^{-2+\alpha+\epsilon'}+C_{k}N_{1}^{-\epsilon'(k-\alpha)}
\end{align*}
and by taking $k$ big enough we can bounds this quantity by $C_{\epsilon'}N_{1}^{-2+\alpha+\epsilon'}$ and in particular by $C_{\epsilon'}N_{1}^{-1+\epsilon'}$.

For the integral over $B_{N_{1},N_{2}}$, this time focusing on the part with $s_{2}>0$ and again applying integration by parts once with respect to $s_{2}$ we have

\begin{align*}
  &|\int_{-N_{1}^{\epsilon'}}^{N_{1}^{\epsilon'}}\int_{r\pi N_{1}}^{r\pi N_{2}} \frac{\cos(\frac{1}{r}s_{2})\cos(s_{1}g'(r))}{|s_{1}^2+s_{2}^2|^{\frac{2-\alpha}{2}}}ds_{2}ds_{1}|\\
  &\leq \frac{C N_{1}^{\epsilon'}}{N_{1}^{2-\alpha}}+C\int_{-N_{1}^{\epsilon'}}^{N_{1}^{\epsilon'}}\int_{r\pi N_{1}}^{r\pi N_{2}}\frac{1}{|s_{1}+s_{2}|^{3-\alpha}}ds_{2}ds_{1}\leq C N^{\epsilon'-2+\alpha}.
\end{align*}

Next we need to show the last equality in \eqref{H*}.
But
\begin{align*}
    &\int_{-r\pi N}^{r\pi N}\int_{-N^{\epsilon'}}^{N^{\epsilon'}} \frac{\cos(\frac{1}{r}s_{2})\cos(s_{1}g'(r))}{|s_{1}^2+s_{2}^2|^{\frac{2-\alpha}{2}}}ds_{1}ds_{2}\\
    &=4\int_{0}^{r\pi N}\int_{0}^{N^{\epsilon'}} \frac{\cos(\frac{1}{r}s_{2}+s_{1}g'(r))}{|s_{1}^2+s_{2}^2|^{\frac{2-\alpha}{2}}}ds_{1}ds_{2}\\
    &=4\int_{0}^{L_{N,r}(A)}\int_{0}^{\frac{\pi}{2}} \frac{\cos(\frac{1}{r}R\sin(A)+R\cos(A)g'(r))}{R^{2-\alpha}}RdAdR
\end{align*}
where we made the change of variables $s_{1}^1+s_{2}^2=R^2$, $A=\arctan\Big(\frac{s_{2}}{s_{1}}\Big)$ and $L_{N,r}(A)$ is  (given values $A,N$ and $r$) the maximum value of $R$ that is still in our domain of integration. The expression for $L_{N,r}(A)$ is complicated but we will just need to use that $L_{N,r}(A)\geq \text{min}(N^{\epsilon'},r\pi N)$. 
Note also that we can rewrite $\frac{1}{r}R\sin(A)+R\cos(A)g'(r)=R\lambda \cos(A+\theta_{0})$ with $\lambda=(\frac{1}{r^2}+g'(r)^2)^{\frac12}$, $\theta_{0}=\arctan(-\frac{1}{rg'(r)})$. But

\begin{align*}
    &\int_{0}^{L_{N,r}(A)}\int_{0}^{\frac{\pi}{2}} \frac{\cos(\lambda R\cos(A+\theta_{0}))}{R^{1-\alpha}}dAdR\\
    &=\int_{0}^{L_{N,r}(\tilde{A}-\theta_{0}-\frac{\pi}{2})}\int_{\delta}^{\frac{\pi}{2}} \frac{\cos(\lambda R\sin(\tilde{A}))}{R^{1-\alpha}}d\tilde{A}dR\\
    &+\int_{0}^{L_{N,r}(\tilde{A}-\theta_{0}-\frac{\pi}{2})}\int_{0}^{\delta}\frac{\cos(\lambda R\sin(\tilde{A}))}{R^{1-\alpha}}dAdR
\end{align*}
so then, using that, for any $S_{2}\geq S_{1}\geq 0$, $\Gamma\neq 0$, since $\alpha\in(0,1)$, we have


\begin{equation*}
    |\int_{0}^{S_{1}} \frac{\cos(\Gamma R)}{R^{1-\alpha}}dR|\leq \Gamma^{-\alpha} C_{max},\  |\int_{S_{1}}^{S_{2}} \frac{\cos(R)}{R^{1-\alpha}}dR|\leq \frac{C_{max}}{S^{1-\alpha}_{1}}
\end{equation*}
we get

$$|\int_{0}^{\delta}\int_{0}^{L_{N,r}(\tilde{A}-\theta_{0}-\frac{\pi}{2})}\frac{\cos(\lambda R\sin(\tilde{A}))}{R^{1-\alpha}}dRd\tilde{A}|\leq \int_{0}^{\delta} C_{max}(\lambda \sin(\tilde{A}))^{-\alpha}d\tilde{A}\leq \frac{CC_{max} \delta}{(\lambda\delta)^{\alpha}}$$
and also

$$\text{lim}_{N\rightarrow\infty}\int_{\delta}^{\frac{\pi}{2}}\int_{0}^{L_{N,r}(\tilde{A}-\theta_{0}-\frac{\pi}{2})} \frac{\cos(\lambda R\sin(\tilde{A}))}{R^{1-\alpha}}dRd\tilde{A}=\int_{\delta}^{\frac{\pi}{2}}P.V.\int_{0}^{\infty} \frac{\cos(\lambda R\sin(\tilde{A}))}{R^{1-\alpha}}dRd\tilde{A}$$
thus


\begin{align*}
    &\frac{H^{^*}}{4}=\text{lim}_{N\rightarrow\infty}\int_{0}^{\frac{\pi}{2}}\int_{0}^{L_{N,r}(A)} \frac{\cos(\lambda R\cos(A+\theta_{0}))}{R^{1-\alpha}}dRdA\\
    &=\text{lim}_{\delta\rightarrow 0}(\text{lim}_{N\rightarrow\infty}\int_{\delta}^{\frac{\pi}{2}}\int_{0}^{L_{N,r}(\tilde{A}-\theta_{0}-\frac{\pi}{2})} \frac{\cos(\lambda R\sin(\tilde{A}))}{R^{1-\alpha}}dRd\tilde{A})\\
    &+\text{lim}_{\delta\rightarrow 0}(\text{lim}_{N\rightarrow\infty}\int_{0}^{\delta}\int_{0}^{L_{N,r}(\tilde{A}-\theta_{0}-\frac{\pi}{2})} \frac{\cos(\lambda R\sin(\tilde{A}))}{R^{1-\alpha}}dRd\tilde{A})\\
    &=\text{lim}_{\delta\rightarrow 0}\int_{\delta}^{\frac{\pi}{2}}P.V.\int_{0}^{\infty} \frac{\cos(\lambda R\sin(\tilde{A}))}{R^{1-\alpha}}dRd\tilde{A}\\
    &=\text{lim}_{\delta\rightarrow 0}\int_{\delta}^{\frac{\pi}{2}}(\lambda \sin(\tilde{A}))^{-\alpha}\int_{0}^{\infty} \frac{\cos(R)}{R^{1-\alpha}}dRd\tilde{A}\\
    &=\frac{1}{(\frac{1}{r^2}+g'(r)^2)^{\frac{\alpha}{2}}}\int_{0}^{\frac{\pi}{2}}\sin(\tilde{A})^{-\alpha}d\tilde{A}\int_{0}^{\infty} \frac{\cos(R)}{R^{1-\alpha}}dR
\end{align*}
so, if we prove that $\int_{0}^{\infty} \frac{\cos(R)}{R^{1-\alpha}}dR$ is positive we are done.

For this, we note that, for $n\in\mathds{N}$

$$\int_{2\pi n}^{2\pi (n+1)}\frac{\cos(R)}{R^{1-\alpha}}dR=\int_{2\pi n}^{2\pi (n+1)}\frac{\sin(R)(1-\alpha)}{R^{2-\alpha}}dR>0$$
where we used integration by parts for the equality and the fact that the denominator is increasing and $\sin(x+\pi)=-\sin(x)$ for the inequality. In fact, the integral is also positive for $n=0$ by taking the integral in $[\delta,2\pi]$, applying integration by parts there and then taking $\delta$ small.

But then since, for $d\leq 2\pi$

$$\int_{2\pi n}^{2\pi n+d}\frac{\cos(R)}{R^{1-\alpha}}dR\leq \frac{C}{(2\pi n)^{1-\alpha}}$$
the limit trivially exists and is positive.

Combining all these bounds we have, for any $\epsilon'>0$, 
$$I_{1}\leq C_{\epsilon'}||f||_{L^{\infty}}N^{(-1-\alpha)(1-\epsilon')},$$ 
so we just need to bound $I_{2}$. In order to bound the $L^2$ norm for $r\notin(\frac14,2)$, we can use integration by parts twice and $h\gtrsim r$, $r+h\approx 1$ to get

\begin{align*}
    &|\int_{-\pi}^{\pi} \frac{f(r+h)\cos(N\tilde{\theta})\cos(N\theta+Ng(r+h)+p(r))}{|h^2+2r(r+h)(1-\cos(\tilde{\theta}))|^{\frac{2-\alpha}{2}}}(r+h)d\tilde{\theta}|\\
    \leq \frac{C||f||_{L^{\infty}}}{N^2}&\int_{-\pi}^{\pi} \frac{r}{|h^2+2r(r+h)(1-\cos(\tilde{\theta}))|^{\frac{2-\alpha}{2}+1}}\big(1+\frac{r}{|h^2+2r(r+h)(1-\cos(\tilde{\theta}))|}\big)d\tilde{\theta}\\
    &\leq \frac{C||f||_{L^{\infty}}}{N^2 |h|^{3-\alpha}}\big(1+\frac{1}{|h|}\big)
\end{align*}
which gives us

$$||(\Lambda^{-\alpha}w(r,\theta))1_{r\in(0,\frac14)}||_{L^{2}}\leq C||(\Lambda^{-\alpha}w(r,\theta))1_{r\in(0,\frac14)}||_{L^{\infty}}\leq \frac{C||f||_{L^\infty}}{N^2}$$
and, for $r\geq 2$
$$|\Lambda^{-\alpha}w(r,\theta)|\leq \frac{C||f||_{L^{\infty}}}{N^2 r^{3-\alpha}}$$
so that 
$$||(\Lambda^{-\alpha}w(r,\theta))1_{r\in(2,\infty)}||_{L^{2}}\leq \frac{C||f||_{L^\infty}}{N^2}$$
which finally gives us
\begin{align*}
    &||\Lambda^{-\alpha}w(r,\theta)-K_{\alpha}\frac{w(r,\theta)}{|(\frac{N}{r})^2+N^2g'(r)^2|^{\alpha/2}}||_{L^{2}}\leq I_{1}+I_{2}\leq C_{\epsilon'}N^{(-1-\alpha)(1-2\epsilon')}.
\end{align*}
Now, taking for example $\epsilon'=\frac{\epsilon}{4}$ finishes the proof for $L^2$. Note also that we only needed the $C^2$ bounds of $f$ and $g$ to obtain this result.
Next to obtain the bound for $H^{3}$ it is enough to check that we have a small $L^2$ norm for some arbitrary third derivative, that is to say, we want to show that

$$||\frac{\partial^3\Lambda^{-\alpha}w(r,\theta)}{\partial x_{i}\partial x_{j}\partial x_{k}}-\frac{\partial^3 K_{\alpha}\frac{w(r,\theta)}{|(\frac{N}{r})^2+N^2g'(r)^2|^{\alpha/2}}}{\partial x_{i}\partial x_{j}\partial x_{k}}||_{L^2}\leq C_{\epsilon,\beta} N^{-1-\alpha+\epsilon+3}||f||_{L^{\infty}}.$$

We will consider $i=j=k=1$ for ease of notation. To bound this norm, we divide it in two different contributions

\begin{align*}
    &||\frac{\partial^3\Lambda^{-\alpha}w(r,\theta)}{\partial x^3_{1}}-\frac{\partial^3 K_{\alpha}\frac{w(r,\theta)}{|(\frac{N}{r})^2+N^2g'(r)^2|^{\alpha/2}}}{\partial x^3_{1}}||_{L^2}\\
    &\leq ||\Lambda^{-\alpha}\frac{\partial^3w(r,\theta)}{\partial x^3_{1}}-\frac{K_{\alpha}}{|(\frac{N}{r})^2+N^2g'(r)^2|^{\alpha/2}}\frac{\partial^3 w(r,\theta)}{\partial x^3_{1}}||_{L^2}\\
    &+||\frac{\partial^2 \frac{K_{\alpha}}{|(\frac{N}{r})^2+N^2g'(r)^2|^{\alpha/2}}}{\partial x^2_{1}}\frac{\partial w(r,\theta)}{\partial x_{1}}||_{L^2}+||\frac{\partial \frac{K_{\alpha}}{|(\frac{N}{r})^2+N^2g'(r)^2|^{\alpha/2}}}{\partial x_{1}}\frac{\partial^2 w(r,\theta)}{\partial x^2_{1}}||_{L^2}\\
    &+||\frac{\partial^3 \frac{K_{\alpha}}{|(\frac{N}{r})^2+N^2g'(r)^2|^{\alpha/2}}}{\partial^3 x_{1}} w(r,\theta)||_{L^2}\\
    &\leq ||\Lambda^{-\alpha}\frac{\partial^3w(r,\theta)}{\partial x^3_{1}}-\frac{K_{\alpha}}{|(\frac{N}{r})^2+N^2g'(r)^2|^{\alpha/2}}\frac{\partial^3 w(r,\theta)}{\partial x^3_{1}}||_{L^2}\\
    &+\sum_{i=1}^{3}N^{-\alpha}|| \frac{K_{\alpha}}{|(\frac{1}{r})^2+g'(r)^2|^{\alpha/2}}||_{C^{i}(r\in(\frac{1}{2},\frac{3}{2}))}||w||_{H^{3-i}}
\end{align*}

The first term of the contribution can be bounded by writing the derivatives in polar coordinates, dividing it in its different frequencies in $\theta$ (which now includes frequencies $N\pm1$, $N\pm2$ and $N\pm 3$, which does not change the bounds for $N$ big ) and using the exact same bounds we obtained in $L^2$. The other term can be bound easily by direct computation by again writing the derivatives in polar coordinates, obtaining the desired bound for $\beta=3$. The interpolation inequality for Sobolev spaces then gives the result for $\beta\in[0,3]$.

Finally, the result for $\lambda>1$ follows directly from a scaling argument plus applying the lemma for $w(\frac{r}{\lambda},\theta)$ since

\begin{align*}
    &||\Lambda^{-\alpha}w(r,\theta)-K_{\alpha}\frac{w(r,\theta)}{|(\frac{N}{r})^2+N^2g'(r)^2|^{\alpha/2}}||_{H^{\beta}}\\
    &\leq \lambda^{-1-\alpha+\beta}||\Lambda^{-\alpha}w(\frac{r}{\lambda},\theta)-K_{\alpha}\frac{w(\frac{r}{\lambda},\theta)}{|(\frac{N}{r})^2+N^2g'(\frac{r}{\lambda})^2|^{\alpha/2}}||_{H^{\beta}}\\
    &\leq C_{P,\epsilon,\alpha} M \lambda^{-1-\alpha+\beta}N^{-1-\alpha+\epsilon+\beta}||f||_{L^{\infty}} 
\end{align*}

\end{proof}

\begin{corollary}\label{+alpha}
For any fixed parameters $\alpha\in(0,1]$, $P,\epsilon>0$ there exists $N_{0}$ such that if $N>N_{0}$, then for any $\lambda\geq 1$ and functions $f(r)$, $g(r)$ and $p(r)$ fulfilling $\text{supp}{f}\subset(\frac{1}{2\lambda},\frac{3}{2\lambda})$ and

$$|| g(\frac{r}{\lambda})||_{C^5}\leq(\ln{(N)})^{P},\ ||f(\frac{r}{\lambda})||_{C^{5}}\leq (\ln{(N)})^{P} ||f||_{L^{\infty}},||p(\frac{r}{\lambda})||_{C^{5}}\leq (\ln{(N)})^{P}$$
then for
$$w(r,\theta):=f(r)\cos(N(\theta+g(r))+p(r))$$
we have that for $\beta\in[0,3-\alpha]$ there exist constants $K_{\alpha}>0$ and $C_{P,\epsilon,\alpha,\beta}$ (depending on $\alpha$ and $P,\epsilon$ and $\alpha$ respectively) such that
$$||\Lambda^{\alpha}w(r,\theta)-K_{\alpha}^{-1}w(r,\theta)|(\frac{N}{r})^2+N^2g'(r)^2|^{\alpha/2}||_{H^{\beta}}\leq C_{P,\epsilon,\alpha} \lambda^{-1+\alpha+\beta}N^{-1+\alpha+\epsilon+\beta}||f||_{L^{\infty}}.$$
\end{corollary}

\begin{proof}
This follows from the previous lemma. If we define, for $w$ as in our statement, the operators

\begin{equation}\label{bar+alpha}
    \bar{\Lambda}^{\alpha}(w(r,\theta)):=K_{\alpha}^{-1}w(r,\theta)|(\frac{N}{r})^2+N^2g'(r)^2|^{\alpha/2}
\end{equation}
\begin{equation}\label{bar-alpha}
    \bar{\Lambda}^{-\alpha}(w(r,\theta)):=K_{\alpha}\frac{w(r,\theta)}{|(\frac{N}{r})^2+N^2g'(r)^2|^{\alpha/2}}
\end{equation}
we have that $\bar{\Lambda}^{-\alpha}(\bar{\Lambda}^{\alpha} w)=w$, $\Lambda^{-\alpha}(\Lambda^{\alpha} w)=w$, so

\begin{align*}
    &(\Lambda^{\alpha}-\bar{\Lambda}^{\alpha})(w)=-\Lambda^{\alpha}(\Lambda^{-\alpha}-\bar{\Lambda}^{-\alpha})\bar{\Lambda}^{\alpha}w
    \end{align*}
and since, for our choice of $w$, $\bar{\Lambda}^{\alpha}w$ fulfils the hypothesis of lemma \ref{-alpha}, we can apply it and get, for $\beta'\in[0,3]$

$$||(\Lambda^{-\alpha}-\bar{\Lambda}^{-\alpha})\bar{\Lambda}^{\alpha}w||_{H^{\beta'}}\leq C_{\epsilon,\alpha,\beta',P} \lambda^{-1+\beta'}N^{-1+\epsilon+\beta'}||f||_{L^{\infty}} $$
and therefore

$$||(\Lambda^{\alpha}-\bar{\Lambda}^{\alpha})w||_{H^{\beta'-\alpha}}=||\Lambda^{\alpha}(\Lambda^{-\alpha}-\bar{\Lambda}^{-\alpha})\bar{\Lambda}^{\alpha}w||_{H^{\beta'-\alpha}}\leq C_{\epsilon,\alpha,\beta',P} \lambda^{-1+\beta'} N^{-1+\epsilon+\beta'}||f||_{L^{\infty}} $$
which finishes the proof by considering $\beta=\beta'-\alpha$.

\end{proof}

\subsection{Other relevant bounds}

Even though the most crucial technical bounds in this paper are the ones we obtained for the fractional dissipation, we need some other technical lemmas in order to obtain a suitable pseudo-solution and control the errors between the pseudo-solution and the real solution to \eqref{SQG}. Corollary \ref{vcuad} and lemma \ref{vr} give useful local approximations for $v$. Lemma \ref{conmutador} and corollary \ref{commutardorlambda} give  commutator estimates for the velocity of a highly oscillatory function, which will be useful when propagating the $L^2$ error between our pseudo-solution and the actual solution to \eqref{SQG}. Lemma \ref{radialdecay} proves some  decay bounds for radial functions solving a fractional heat equation, and finally lemma \ref{initialvel} shows that we can find a radial function with several useful properties, that we will use to construct the initial conditions for our pseudo-solution.

\begin{corollary}\label{vcuad}
For any fixed parameters $P,\epsilon>0$ there exists $N_{0}$ such that if $N>N_{0}$, then for any $\lambda\geq 1$ and  functions $f(r)$, $g(r)$ and $p(r)$ fulfilling $\text{supp} f(r)\subset(\frac{1}{2\lambda},\frac{3}{2\lambda})$ and

$$|| g(\frac{r}{\lambda})||_{C^5}\leq (\ln{(N)})^{P}\ ||f(\frac{r}{\lambda})||_{C^{5}}\leq (\ln{(N)})^{P} ||f||_{L^{\infty}}, ||p(\frac{r}{\lambda})||_{C^{5}}\leq (\ln{(N)})^{P}$$
if we define
$$w(r,\theta):=f(r)\cos(N(\theta+g(r))+p(r))$$
we have that for $\beta\in[0,2]$ there exist a constant $C_{\epsilon,\beta,P}$ (depending on $\beta,P$ and $\epsilon$) such that

$$||v_{1}(w(r,\theta))+\frac{\partial}{ \partial x_{2}}\bar{\Lambda}^{-1}(w)(r,\theta)||_{H^{\beta}}\leq C_{\epsilon,\beta,P}  N^{-1+\epsilon+\beta}||f||_{L^{\infty}},$$

$$||v_{2}(w(r,\theta))-\frac{\partial}{ \partial x_{1}}\bar{\Lambda}^{-1}(w)(r,\theta)||_{H^{\beta}}\leq C_{\epsilon,\beta,P}  N^{-1+\epsilon+\beta}||f||_{L^{\infty}},$$
with $\bar{\Lambda}^{-1}$ defined as in \eqref{bar-alpha}.

\end{corollary}

\begin{proof}
    This is a direct consequence of lemma \ref{-alpha}, since 
    $$v_{1}(w(r,\theta))=-\frac{\partial}{\partial x_{2}}\Lambda^{-1}(w(r,\theta)),\quad v_{2}(w(r,\theta))=\frac{\partial}{\partial x_{1}}\Lambda^{-1}(w(r,\theta))$$
    and applying lemma \ref{-alpha}, we have
    \begin{align*}
        &||\frac{\partial}{\partial x_{i}}\Lambda^{-1}(w(r,\theta))-\frac{\partial}{\partial x_{i}}\bar{\Lambda}^{-1}(w(r,\theta))||_{H^{\beta}}\leq ||\Lambda^{-1}(w(r,\theta))-\bar{\Lambda}^{-1}(w(r,\theta))||_{H^{\beta+1}}\\
        &\leq C_{\epsilon,\beta,P} \lambda^{-1+\beta}N^{-1+\epsilon+\beta}||f||_{L^{\infty}}.
    \end{align*}
\end{proof}

\begin{lemma}\label{vr}
For any fixed parameters $P,\epsilon>0$ there exists $N_{0}$ such that if $N>N_{0}$, then for any $\lambda\geq 1$ and functions $f(r)$,$g(r)$ and $p(r)$ fulfilling $\text{supp}f(r)\subset(\frac{1}{2\lambda},\frac{3}{2\lambda})$ and

$$|| g(\frac{r}{\lambda})||_{C^5}\leq (\ln{(N)})^{P}\ ||f(\frac{r}{\lambda})||_{C^{5}}\leq (\ln{(N)})^{P} ||f(\frac{r}{\lambda})||_{L^{\infty}},||p(\frac{r}{\lambda})||_{C^{5}}\leq (\ln{(N)})^{P}$$
if we define
$$w(r,\theta):=f(r)\cos(N(\theta+g(r))+p(r))$$
we have that for $\beta\in[0,2]$ there exist constants $C_{0}>0$ and $C_{P,\beta,\epsilon}$ (depending on $P,$ $\beta$ and $\epsilon$) such that for $N$ big enough,

$$||v_{r}(w(r,\theta))+C_{0}\frac{1}{(1+r^2g'(r)^2)^{\frac12}}f(r)\sin(N(\theta+g(r))+p(r))||_{H^{\beta}}\leq C_{\epsilon,\beta,P} \lambda^{-1+\beta} N^{-1+\epsilon+\beta}||f||_{L^{\infty}}$$
where $v_{r}:=v\cdot\hat{r}$ is the radial component of the velocity.

\end{lemma}

\begin{proof}
This proof is very similar to that of lemma \ref{-alpha}, but now we study the operator
\begin{align*}
&v_{r}(w)(r,\theta)=v\cdot\hat{r}\\&= \int_{[-r,\infty]\times[-\pi,\pi]}\frac{(r+h)^2\sin(\theta')(w(r+h,\theta'+\theta)-w(r,\theta))}{|h^2+2r(r+h)(1-\cos(\theta'))|^{3/2}}d\theta'dh
\end{align*}
so we will not delve too deeply into the details and mostly mention the key differences. Again, as before, we consider $P=1$ for simplicity, and we start dealing with the case $\beta=0,\lambda=1$.
First, using integration by parts $k$ times with respect to $\theta'$ as in lemma \ref{-alpha}, we note that, for $r\in(\frac14,2)$ we have that, 
$$ \int_{(-r,\infty)\setminus [-N^{-1+\epsilon'},N^{-1+\epsilon'}]}\int_{-\pi}^{\pi}\frac{(r+h)^2\sin(\theta')(w(r+h,\theta'+\theta)-w(r,\theta))}{|h^2+2r(r+h)(1-\cos(\theta'))|^{3/2}}d\theta'dh|\leq C_{\epsilon'}N^{-1}||f||_{L^{\infty}},$$
and furthermore, integrating by parts $2$ times  with respect to $\theta'$ we obtain

\begin{align*}
    &|\int_{-N^{-1+\epsilon'}}^{N^{-1+\epsilon'}}\int_{[N^{-\frac12},\pi]\cup[-\pi,-N^{-\frac12}]}\frac{(r+h)^2\sin(\theta')(w(r+h,\theta'+\theta)-w(r,\theta))}{|h^2+2r(r+h)(1-\cos(\theta'))|^{3/2}}d\theta'dh|\\
    &= |\int_{[N^{-\frac{1}{2}},\pi]\cup[-\pi,-N^{-\frac12}]}\int_{-N^{-1+\epsilon'}}^{N^{-1+\epsilon'}} \frac{f(r+h)\sin(N\theta')\sin(\theta')\sin(N\theta+Ng(r+h)+p(r+h))}{|h^2+2r(r+h)(1-\cos(\theta'))|^{\frac{3}{2}}}(r+h)^2dhd\theta'|\\
    &\leq C\Big(\frac{1}{N}\int_{-N^{-1+\epsilon'}}^{N^{-1+\epsilon'}} \frac{||f||_{L^{\infty}}}{|(1-\cos(N^{-\frac12}))|}dh+\int_{N^{-\frac12}}^{\pi}\int_{-N^{-1+\epsilon'}}^{N^{-1+\epsilon'}}\frac{1}{N^{2}} \frac{||f||_{L^{\infty}}}{|(1-\cos(\theta'))|^{2}}dhd\theta'\Big)\\
    &\leq C||f||_{L^{\infty}}N^{-1+\epsilon'}.
\end{align*}
So we can focus on the integral over $A:=[-N^{-1+\epsilon},N^{-1+\epsilon}]\times[-N^{-\frac12},N^{-\frac12}]$. Then, we check that 

\begin{align*}
     &|\int_{A}\Big(\frac{(r+h)^2\sin(\theta')(w(r+h,\theta'+\theta)-w(r,\theta))}{|h^2+2r(r+h)(1-\cos(\theta'))|^{3/2}}-\frac{r^2\theta'(w(r+h,\theta'+\theta)-w(r,\theta))}{|h^2+r^2 (\theta')^2)|^{3/2}}\Big)d\theta'dh|\\
     &\leq C||f||_{L^{\infty}}N^{-1+\epsilon'}
\end{align*}
so that we can work with the simplified version of the kernel, and we also have

\begin{align*}
    &|\int_{A}\frac{r^2\theta'(w(r+h,\theta'+\theta)-f(r)\sin(N\theta')\sin(N\theta+Ng(r)+Nhg'(r)+p(r)))}{|h^2+r^2 (\theta')^2)|^{3/2}}d\theta'dh|\\
    &\leq C||f||_{L^\infty}N^{-1+3\epsilon'}.
\end{align*}
Altogether, we obtain, for $r\in(\frac14,2)$,

$$|v_{r}(w)+f(r)\sin(N\theta+Ng(r)+p(r))\int_{A}\frac{r^2\theta'\sin(N\theta'+Nhg'(r)))}{|h^2+r^2 (\theta')^2)|^{3/2}}d\theta'dh|\leq C_{\epsilon'}N^{-1+3\epsilon'}||f||_{L^{\infty}}.$$
Next, defining $H_{N}$ as 

\begin{align*}
    &H_{N}:=\int_{A}\frac{r^2\theta'\sin(N\theta'+Nhg'(r)))}{|h^2+r^2 (\theta')^2)|^{3/2}}d\theta'dh=\int_{-\frac{N^{\epsilon'}}{r}}^{\frac{N^{\epsilon'}}{r}}\int_{-N^{\frac12}}^{N^{\frac12}}\frac{ s_{1}\sin(s_{1})\cos(rs_{2}g'(r))}{|s_{1}^{2}+s_{2}^2|^{3/2}}ds_{1}ds_{2}.
\end{align*}

We want to show that 

\begin{equation}\label{relacionH*}
   H^{*}:=\text{lim}_{N\rightarrow\infty}H_{N}=\frac{C_{0}}{|1+r^2g'(r)^2|^{1/2}}
\end{equation} 
and that 

\begin{equation}\label{HN2}
    |H^{*}-\text{lim}_{N\rightarrow\infty} H_{N}|\leq C_{\epsilon'}N^{-1+\epsilon'}||f||_{L^{\infty}}.
\end{equation}
(\ref{HN2}) is obtained exactly as in lemma \ref{-alpha}, by getting bounds for $|H_{N_{2}}-H_{N_{1}}|$ using integration by parts in the domain that  remains after canceling out the integrals from $H_{N_{1}}$ and $H_{N_{2}}$. As for (\ref{relacionH*}), using that, for any $K\in\mathds{R}$

$$|\int_{0}^{K}\frac{\sin(x)}{x}dx|\leq C,$$
and that, for $\lambda\neq 0$
$$\int_{0}^{\infty}\frac{\sin(\lambda x)}{x}dx=\text{sign}(\lambda)\int_{0}^{\infty}\frac{\sin(x)}{x}dx=\text{sign}(\lambda)\frac{C_{0}}{4}$$
we get that, using the change of variables $s_{1}=R\sin(A)$, $s_{2}=R\cos(A)$, basic trigonometric identities, checking carefully  the convergence of the integrals, and using 
$$\sin(\tan^{-1}(x))=\frac{x}{(1+x^2)^{\frac{1}{2}}}$$
we get
\begin{align*}
    &H^{*}=\text{lim}_{N\rightarrow \infty}\int_{-\frac{N^{\epsilon'}}{r}}^{\frac{N^{\epsilon}}{r}}\int_{-N^{\frac12}}^{N^{\frac12}}\frac{ s_{1}\sin(s_{1})\cos(rs_{2}g'(r))}{|s_{1}^{2}+s_{2}^2|^{3/2}}ds_{1}ds_{2}\\
    &=\int_{-\pi}^{\pi}\sin(A)\int_{0}^{\infty}\frac{ \sin(R(\sin(A)+\cos(A)rg'(r)))}{R}dRdA\\
    &=\frac{C_{0}}{4}\int_{-\pi}^{\pi}\sin(A) \text{sign}(\sin(A)+\cos(A)rg'(r))dA\\
    &=\frac{C_{0}}{4}\int_{-\pi}^{\pi}\sin(A) \text{sign}(\cos(A+\tan^{-1}(\frac{-1}{rg'(r)})))dA\\
    &=-\frac{C_{0}}{4}\sin(\tan^{-1}(\frac{-1}{rg'(r)}))\int_{-\pi}^{\pi}\cos(A+\tan^{-1}(\frac{-1}{rg'(r)}))) \text{sign}(\cos(A+\tan^{-1}(\frac{-1}{rg'(r)})))dA\\
    &=C_{0}\frac{1}{(1+r^2g'(r)^2)^{\frac12}}
\end{align*}
as we wanted to prove. The rest of the proof does not have any meaningful differences with lemma \ref{-alpha}, a bound for the decay of $v_{r}(w)$ is obtained for $r\notin (\frac14,2)$ to obtain the $L^2$ bound and then taking three derivatives, applying Leibniz rule and bounding each term gives us the bound for $H^3$. Then, the interpolation inequality gives then the result for $\beta\in(0,3)$.

For the case $\lambda>1$, 

\begin{align*}
    &||v_{r}(w(r,\theta))+C_{0}\frac{1}{(1+r^2g'(r)^2)^{\frac12}}f(r)\sin(N(\theta+g(r))+p(r))||_{H^{\beta}}\\
    &\leq \lambda^{-1+\beta}||v_{r}(w(\frac{r}{\lambda},\theta))+C_{0}\frac{1}{(1+r^2g'(\frac{r}{\lambda})^2)^{\frac12}}f(\frac{r}{\lambda})\sin(N(\theta+g(\frac{r}{\lambda}))+p(\frac{r}{\lambda}))||_{H^{\beta}}\\
    &\leq C_{\epsilon,\beta,P} M \lambda^{-1+\beta} N^{-1+\epsilon+\beta}||f||_{L^{\infty}}.
\end{align*}

\end{proof}

\begin{lemma} \label{conmutador}
Given $\epsilon>0$, $N\in\mathds{N}>3$, then for any functions $g(r),f(r),p(r)\in C^{1}$, $f(r)\in L^2$ and $g(r)$ with support in $r\in[1,4]$,
if we define
$$w(r,\theta):=f(r)\cos(N\theta+p(r))$$
then we have that for $i=1,2$
$$||[v_{i}(g(r)\sin(\theta)w(r,\theta))-g(r)\sin(\theta)v_{i}(w(r,\theta))]||_{L^{2}}\leq C_{\epsilon} N^{-1+\epsilon}||g||_{C^{1}}||w||_{L^{2}},$$
$$||[v_{i}(g(r)\cos(\theta)w(r,\theta))-g(r)\cos(\theta)v_{i}(w(r,\theta))]||_{L^{2}}\leq C_{\epsilon} N^{-1+\epsilon}||g||_{C^{1}}||w||_{L^{2}}.$$
\end{lemma}

\begin{proof}

We will only consider the first inequality and with $i=1$, since the other cases are analogous. We note now that

$$v_{1}(W(r,\theta))=P.V.\int_{-r}^{\infty}\int_{-\pi}^{\pi}\frac{(r+h)\sin(\theta+\theta')-r\sin(\theta)}{|h^2+2r(r+h)(1-\cos(\theta'))|^{\frac32}}(r+h)W(r+h,\theta+\theta')dhd\theta'.$$

Now, since the principal value integral is defined using cartesian coordinates, we would like to show that for $C^1$ functions, there is a more suitable expression using polar coordinates, namely we would like to show that

\begin{align*}
    &\text{lim}_{\epsilon\rightarrow 0}\Big(\int_{\mathds{R^2}\setminus B_{\epsilon}(x)}\frac{y_{2}-x_{2}}{|x-y|^{3}}w(y)dy-\int_{||x|-|y||\geq \epsilon}\frac{y_{2}-x_{2}}{|x-y|^{3}}w(y)dy\Big)\\
    &=\text{lim}_{\epsilon\rightarrow 0}\int_{\{|x|+\epsilon\geq |y|\geq |x|-\epsilon\}\setminus B_{\epsilon}(x)}\frac{y_{2}-x_{2}}{|x-y|^{3}}w(y)dy=0
\end{align*}
but, writing the integrals in polar coordinates and cancelling all the terms with the wrong parity with respect to $h$ or $\theta'$

\begin{align*}
    &|\int_{\{|y|+\epsilon\geq |x|\geq |y|-\epsilon\}\setminus B_{\epsilon}(x)}\frac{y_{2}-x_{2}}{|x-y|^{3}}w(x)dy|\\
    &\leq C||w||_{L^{\infty}}\int_{\{|h|\geq \epsilon\}\setminus B_{\epsilon}(x)}\frac{h^2+\theta^2}{|h^2+2r(r+h)(1-\cos(\theta'))|^{\frac32}}dhd\theta' \leq C||w||_{L^{\infty}} \epsilon|\ln(\epsilon)|
\end{align*}
and
\begin{align*}
    &|\int_{|y|+\epsilon\geq |x|\geq |y|-\epsilon\setminus B_{\epsilon}(x)}\frac{y_{2}-x_{2}}{|x-y|^{3}}(w(y)-w(x))dy| \\
    &\leq||w||_{C^1}|\int_{|y|+\epsilon\geq |x|\geq |y|-\epsilon\setminus B_{\epsilon}(x)}\frac{1}{|x-y|}dy|\leq C||w||_{C^1}\epsilon|\ln(\epsilon)|,
\end{align*}
so in particular we can write

$$v_{1}(W(r,\theta))=P.V.\int_{-r}^{\infty}\int_{-\pi}^{\pi}\frac{(r+h)\sin(\theta+\theta')-r\sin(\theta)}{|h^2+2r(r+h)(1-\cos(\theta'))|^{\frac32}}(r+h)W(r+h,\theta+\theta')dhd\theta'$$
$$=\text{lim}_{\epsilon\rightarrow 0}\int_{[-r,\infty]\setminus [-\epsilon,\epsilon]}\int_{-\pi}^{\pi}\frac{(r+h)\sin(\theta+\theta')-r\sin(\theta)}{|h^2+2r(r+h)(1-\cos(\theta'))|^{\frac32}}(r+h)W(r+h,\theta+\theta')dhd\theta'.$$

Now,  using integration by parts $k$ times with respect to $\theta'$, we have that, for $ |h|\geq \frac{1}{2}$, $(r+h)\in(1,4)$

\begin{align*}
    &|\int_{-\pi}^{\pi}\frac{(r+h)\sin(\theta+\theta')-r\sin(\theta)}{|h^2+2r(r+h)(1-\cos(\theta'))|^{3/2}}(r+h)\cos(N\theta'+C)d\theta'|\\
    &\leq |\int_{-\pi}^{\pi}\frac{rC_{k}}{N^{k}|h^2+2r(r+h)(1-\cos(\theta'))|^{3/2}}d\theta'|\\
    &\leq \frac{rC_{k}}{h^3N^{k}}.
\end{align*}

Analogously if $\frac12\geq |h|\geq N^{-1+\epsilon}$, we can also perform integration by parts $k$ times

\begin{align*}
     & |\int_{-\pi}^{\pi}\frac{(r+h)\sin(\theta+\theta')-r\sin(\theta)}{|h^2+2r(r+h)(1-\cos(\theta'))|^{3/2}}(r+h)\cos(N\theta'+C)d\theta'|\\
    &\leq  \frac{rC_{k}}{N^{k}}\int_{-\pi}^{\pi}\sum_{i=0}^{k}\frac{1}{|h^2+2r(r+h)(1-\cos(\theta'))|^{\frac {3+i}{2}}}d\theta'\\
    &\leq \frac{C_{k}rN^{(1-\epsilon)k}}{h^3 N^{k}}\leq \frac{C_{k}rN^{-k \epsilon}}{h^3 }
\end{align*}
for any $k\in\mathds{N}$ for some $C_{k}$.


Using this, for any $H(r)\in L^2$ 
and $p(r)\in C^1$  if we define $v_{1,1}$

$$
\mbox{\small $v_{1,1}(H(r)\cos(N\theta+p(r))):=\int_{ |h|\geq N^{-1+\epsilon'}}\int_{-\pi}^{\pi}\frac{(r+h)\sin(\theta+\theta')-r\sin(\theta)}{|h^2+2r(r+h)(1-\cos(\theta'))|^{3/2}}(r+h)\cos(N\theta+p(r))H(r+h)d\theta' dh$ }
$$
we get, if $r\in[1,4]$,
\begin{align}\label{H1}
    &|v_{1,1}(H(r)\cos(N\theta+p(r)))|\leq \int_{|h|\geq N^{-1+\epsilon'}} (\frac{C_{k}r}{h^3 N^{k}}+\frac{C_{k}r}{h^3N^{k\epsilon}}) |H(r+h)|dh\\
    &\leq \frac{C_{\epsilon'}}{N}||H||_{L^2}\nonumber
\end{align}
and, if $\text{supp}H(r)\subset r\in[1,4]$
\begin{align}\label{H2}
    &|v_{1,1}(H(r)\cos(N\theta+p(r)))|\leq \int_{|h|\geq N^{-1+\epsilon'}} (\frac{C_{k}r}{h^3 N^{k}}+\frac{C_{k}r}{h^3N^{k\epsilon}}) |H(r+h)|dh\\
    &\leq \frac{C_{\epsilon'}}{N}\text{min}(1,\frac{1}{r^2})||H||_{L^2}.\nonumber
\end{align}
We can then apply \eqref{H1} to obtain $||g(r)\sin(\theta)v_{1,1}(w(r,\theta))||_{L^{2}}\leq \frac{C_{\epsilon'}}{N}||g||_{L^\infty}||w||_{L^2}$ and by decomposing $\sin(\theta)w(r,\theta)$ in its different Fourier modes and applying \eqref{H2} we get $||v_{1,1}(g(r)\sin(\theta)w(r,\theta))||_{L^{2}}\leq \frac{C_{\epsilon',A}}{N}||g||_{L^{\infty}}||w||_{L^2}$  .

If we now further divide the operator $v_{1}(W)$ as $v_{1}(W)=v_{1,1}(W)+v_{1,2}(W)+v_{1,3}(W)$, with

\begin{align*}
    &v_{1,2}(W):=\int_{-N^{-1+\epsilon'}}^{N^{-1+\epsilon'}}\int_{-N^{-1+\epsilon'}}^{N^{-1+\epsilon'}}\frac{(r+h)\sin(\theta+\theta')-r\sin(\theta)}{|h^2+2r(r+h)(1-\cos(\theta'))|^{\frac32}}(r+h)W(r+h,\theta+\theta')dhd\theta'
\end{align*}
\begin{align*}
    &v_{1,3}(W):=\int_{N^{-1+\epsilon'}}^{2\pi-N^{-1+\epsilon'}}\int_{-N^{-1+\epsilon'}}^{N^{-1+\epsilon'}}\frac{(r+h)\sin(\theta+\theta')-r\sin(\theta)}{|h^2+2r(r+h)(1-\cos(\theta'))|^{\frac32}}(r+h)W(r+h,\theta+\theta')dhd\theta'
\end{align*}
using the previous bound it is enough to show that, for $i=2,3$

$$||[v_{1,i}(g(r)\sin(\theta)w(r,\theta))-g(r)\sin(\theta)v_{1,i}(w(r,\theta))]||_{L^{2}} \leq C_{\epsilon} N^{-1+\epsilon}||g||_{C^{1}}||f||_{L^{2}}.$$

But, for $i=2$
\begin{align*}
    &||g(r)\sin(\theta) v_{1,2}(w(r,\theta))-v_{1,2}(g(r)\sin(\theta)w(r,\theta))||^2_{L^2}\\
    &\leq ||\int_{-N^{-1+\epsilon'}}^{N^{-1+\epsilon'}}\int_{-N^{-1+\epsilon'}}^{N^{-1+\epsilon'}}\frac{(r+h)\sin(\theta+\theta')-r\sin(\theta)}{|h^2+2r(r+h)(1-\cos(\theta'))|^{\frac32}}(r+h)\\
    &[g(r+h)\sin(\theta+\theta')-g(r)\sin(\theta)]w(r+h,\theta+\theta')dhd\theta'||^{2}_{L^2}\\
    &\leq C||g||^2_{C^1}\int_{0}^{5}\int_{0}^{2\pi}\Big(\int_{|h|\leq N^{-1+\epsilon'}}\int_{|\theta'|\leq N^{-1+\epsilon'}}\frac{(r+h)^2(|\theta'|+|h|)^2|f(r+h)|}{|h^2+2r(r+h)(1-\cos(\theta'))|^{3/2}}d\theta' dh\Big)^2rdrd\theta
\end{align*}

\begin{align*}
    &\leq C||g||^2_{C^1}\int_{0}^{5}\Big(\int_{|h|\leq N^{-1+\epsilon'}}\int_{|\theta'|\leq N^{-1+\epsilon'}}\frac{|f(r+h)|}{|h^2+r^2(\theta')^2|^{1/2}}d\theta' dh\Big)^2rdr\\
    &\leq C||g||^2_{C^1}\int_{|h_{1}|\leq N^{-1+\epsilon'}}\int_{|\theta_{1}'|\leq N^{-1+\epsilon'}}\frac{1}{|h_{1}^2+\theta_{1}^2|^{1/2}}\int_{|h_{2}|\leq N^{-1+\epsilon'}}\int_{|\theta_{2}'|\leq N^{-1+\epsilon'}}\frac{1}{|h_{2}^2+\theta_{2}^2|^{1/2}}\\
    &\int_{0}^{2}|f(r+h_{1})||f(r+h_{2})|rdrd\theta_{1}' dh_{1}d\theta_{2}' dh_{1}\\
    &\leq CN^{-2+2\epsilon'}||g||^2_{C^1}||f||^2_{L^2}
\end{align*}
and similarly

\begin{align*}
    &||f(r)v_{1,3}(g(r)\cos(N\theta))-v_{1,3}(f(r)g(r)\cos(N\theta))||^2_{L^{2}}\\
    &\leq C||g||^2_{C^1}\int_{0}^{5}\Big(\int_{|h|\leq N^{-1+\epsilon'}}\int_{\pi\geq|\theta'|\geq N^{-1+\epsilon'}}\frac{|f(r+h)|}{|h^2+(\theta')^2|^{1/2}}d\theta' dh\Big)^2rdr\\
    &\leq C||g||^2_{C^1}\int_{0}^{2}\Big(\int_{|h|\leq N^{-1+\epsilon'}}|f(r+h)|\ln(|h|+N^{-1+\epsilon}) dh\Big)^2rdr\\
    &\leq C||g||^2_{C^1}\int_{|h_{1}|\leq N^{-1+\epsilon'}}\ln(|h_{1}|+N^{-1+\epsilon}) \int_{|h_{2}|\leq N^{-1+\epsilon'}}\ln(|h_{2}|+N^{-1+\epsilon})\int_{0}^{2}|f(r+rh_{1})||f(r+h_{2})|rdrdh_{2}dh_{1}\\
    &\leq C(\ln(N))^2N^{-2+2\epsilon'}||g||^2_{C^1}||f||^2_{L^2}.
\end{align*}

We obtain now our result from combining all our inequalities since

\begin{align*}
    &||v_{1}(f(r)g(r)\cos(N\theta))-f(r)v_{1}(g(r)\cos(N\theta))||_{L^{2}}\\
    &\leq \sum_{i=1}^{3}||v_{1,i}(f(r)g(r)\cos(N\theta))-f(r)v_{1,i}(g(r)\cos(N\theta))||_{L^{2}}\\
    &\leq C_{\epsilon'}||f||_{L^{2}}||g||_{C^{1}}(N^{-1}+N^{-1+\epsilon'}+\ln{(N)}N^{-1+\epsilon'}).
\end{align*}

\end{proof}

\begin{remark}
    Although we require $f(r),h(r)\in C^1$, note that the bounds do not involve their $C^1$ norms, since this condition is only required to ensure convergence of the integrals involved.
\end{remark}

\begin{corollary} \label{commutardorlambda}
Given $\epsilon>0$, $N\in\mathds{N}>3$, $\lambda>0$, then for any functions $g(r),f(r),h(r)\in C^{1}$, $f(r)\in L^2$ and $g(r)$ with support in $r\in[\lambda^{-1},4\lambda^{-1}]$,
we have that for $i=1,2$, if we define
$$w(r,\theta):=f(r)\cos(N\theta+h(r))$$

$$||[v_{i}(g(r)\sin(\theta)w(r,\theta))-g(r)\sin(\theta)v_{i}(w(r,\theta))]||_{L^{2}}\leq C_{\epsilon}N^{-1+\epsilon}||g(\frac{r}{\lambda})||_{C^{1}}||w||_{L^{2}},$$
$$||[v_{i}(g(r)\cos(\theta)w(r,\theta))-g(r)\cos(\theta)v_{i}(w(r,\theta))]||_{L^{2}}\leq C_{\epsilon}  N^{-1+\epsilon}||g(\frac{r}{\lambda})||_{C^{1}}||w||_{L^{2}}.$$


\end{corollary}

\begin{proof}
    To prove it is enough to note that, for $f(r)$, $g(r)$ as in our hypothesis, $f(\frac{r}{\lambda}),g(\frac{r}{\lambda})$ would fulfil the hypothesis of lemma \ref{conmutador}, so again focusing on $i=1$ and $g(r)\sin(\theta)$, using the scaling properties of $v$ and of the $L^2$ and $C^1$ norms we have

\begin{align*}
    &||v_{1}(g(r)\sin(\theta)w(r,\theta))-g(r)\sin(\theta)v_{1}(w(r,\theta))||_{L^{2}}\\
    &=\frac{1}{\lambda} ||v_{1}(g(\frac{r}{\lambda})\sin(\theta)w( \frac{r}{\lambda},\theta))-g( \frac{r}{\lambda})\sin(\theta)v_{1}(w(\lambda r,\theta))||_{L^{2}}\\
    &\leq \frac{C_{\epsilon}}{\lambda}  N^{-1+\epsilon}||g(\frac{ r}{\lambda})||_{C^1}||w(\frac{r}{\lambda},\theta)||_{L^2}\leq \frac{C_{\epsilon}}{\lambda}N^{-1+\epsilon}||g(\frac{r}{\lambda})||_{C^1}\lambda||w(r,\theta)||_{L^2}\\
    &=C_{\epsilon}  N^{-1+\epsilon}||g(\frac{r}{\lambda})||_{C^1}||w(r,\theta)||_{L^2}.
\end{align*}
    
\end{proof}

\begin{lemma}\label{radialdecay}
    Let $\alpha\in(0,1)$, $f(r)$ be a $C^{\infty}$ radial function with $\hat{f}\in C^{\infty}$. Let $f(r,t)$ be the solution to
    $$\p_{t}f(r,t)=-\Lambda^{\alpha}f(r,t),$$
    $$f(r,0)=f(r).$$
    Then, for any $\epsilon>0$, $T>0$, there exists a constant $C_{\epsilon}$ such that, for $t\in [0,T]$ we have that
    $$|(\p_{r}f)(r_{0},t)|\leq C_{\epsilon}(1+T)r_{0}^{-\frac{3+2\alpha-\epsilon}{3}},$$
    $$|(\p_{r}\p_{r})f(r_{0},t)|\leq C_{\epsilon}(1+T)r_{0}^{-\frac{5+2\alpha-\epsilon}{3}}.$$
\end{lemma}

\begin{proof}
    We will prove the bounds for $(\p_{r}f)(r_{0},t)$, the bounds for the second derivative being completely analogous.
    We start by noting that $\widehat{f(\cdot,t)}=\hat{f}(|\xi|)e^{-|\xi|^{\alpha}t}$. Furthermore, we have that
    $$\widehat{\p_{x_{1}}f(\cdot,t)}=-2\pi i\xi_{1}\hat{f}(|\xi|)e^{-|\xi|^{\alpha}t}.$$
    Since 
    $$||2\pi\xi_{1}\hat{f}(|\xi|)e^{-|\xi|^{\alpha}t}||_{C^{1,\alpha}}\leq C(1+t)$$
    $$||2\pi\xi_{1}\hat{f}(|\xi|)e^{-|\xi|^{\alpha}t}||_{L^2}\leq ||f||_{H^2}\leq C$$
    we have that, for any $\epsilon>0$, $t\in[0,T]$,
    $$||2\pi\xi_{1}\hat{f}(|\xi|)e^{-|\xi|^{\alpha}t}||_{H^{1+\alpha-\epsilon}}\leq C_{\epsilon}(1+T)$$
    which implies
    $$C||r^{1+\alpha-\epsilon}\p_{x_{1}}f(r,t)||_{L^2}\leq C_{\epsilon}(1+T).$$
    But then, using the notation $\tilde{f}(r):=\p_{x_{1}}f(r)$, $M=||\tilde{f}(r)||_{C^1}<\infty$
    \begin{align*}
        &||r^{1+\alpha-\epsilon}\tilde{f}(r)||^2_{L^2}\geq \int_{r_{0}}^{r_{0}+\frac{\tilde{f}(r_{0})}{2M}}(r^{1+\alpha-\epsilon}\tilde{f}(r)^2)^2rdr\geq r_{0}^{3+2\alpha-2\epsilon}\int_{r_{0}}^{r_{0}+\frac{\tilde{f}(r_{0})}{2M}}\tilde{f}(r)^2dr\\
        &\geq r_{0}^{3+2\alpha-2\epsilon}\frac{\tilde{f}(r_{0})^2}{4}\int_{r_{0}}^{r_{0}+\frac{\tilde{f}(r_{0})}{2M}}dr\geq r_{0}^{3+2\alpha-2\epsilon}\frac{\tilde{f}(r_{0})^3}{8M}
    \end{align*}
    and thus
    $$C_{\epsilon}(1+T)\geq r_{0}^{3+2\alpha-2\epsilon}\frac{\tilde{f}(r_{0})^3}{8M}$$
    and since $$\p_{r}f(r)=\cos(\theta)\p_{x_{1}}f(r)$$
    this finishes the proof.
\end{proof}

Before we can define our pseudo-solution, we need one last technical lemma.
\begin{lemma}\label{initialvel}
Given  any $a_{i}$ for $i=1,2,$ there exists a  $C^{\infty}$ function $g(r)$ such that $\hat{g}(\hat{r})\in C^{\infty}$ has support in $\hat{r}\in(c,\infty)$ for some $c>0$ and 

$$\frac{\partial}{\partial r}\frac{v_{\theta}(g(r))}{r}(r=1)=a_{1}$$
$$\frac{\partial^2}{\partial r^2}\frac{v_{\theta}(g(r))}{r}(r=1)=a_{2},$$
$$\frac{\partial^3}{\partial r^3}\frac{v_{\theta}(g(r))}{r}(r=1)=a_{3},$$
where $v_{\theta}$ is the angular component of the velocity, $\hat{g}$ is the Fourier transform of $g$ and $\hat{r}$ is the radial variable in the frequency domain.
    
\end{lemma}

\begin{proof}

We start by choosing $h_{1}(r)$ smooth function with support in $r\in (\frac14,\frac12)$ fulfilling

$$\int_{0}^{\frac12}sh_{1}(s)ds=1$$
and then we define

$$h_{2}(r):=\frac{1}{r}\frac{\partial^2}{\partial r^2}rh_{1}(r)$$
$$h_{3}(r):=\frac{1}{r}\frac{\partial^4}{\partial r^4}rh_{1}(r).$$

Now, since for a generic radial function $h(r)$ we have

$$v_{\theta}(h(\cdot))(r,\alpha)= P.V.\int_{\mathds{R}_{+}\times[-\pi,\pi]}r'\frac{r-r'\cos(\theta')}{|r^2+(r')^2-2rr'\cos(\theta'))|^{3/2}}(h(r')-h(r))d\theta'dr'$$
then for $r\in(\frac12,\frac32)$ we have
\begin{equation}\label{h1}
    \text{lim}_{\lambda\rightarrow\infty}v_{\theta}(\lambda^2h_{1}(\lambda\cdot))=\text{lim}_{\lambda\rightarrow\infty}2\pi\int_{\mathds{R}_{+}}\frac{r}{|r^2|^{3/2}}\lambda^2r'h_{1}(\lambda r')dr'=\frac{2\pi}{r^2}
\end{equation}
and furthermore there is strong convergence in $C^{k}$ for $r\in(\frac12,\frac32)$ for any fixed $k$. On the other hand, using integration by parts twice with respect to $r'$, we have that, for $r\in(\frac12,\frac32)$, $\lambda>3$

\begin{align*}
    v_{\theta}(\lambda^4h_{2}(\lambda\cdot))(r,\alpha)&=P.V.\int_{\mathds{R}_{+}\times[-\pi,\pi]}\frac{r-r'\cos(\theta')}{|r^2+(r')^2-2rr'\cos(\theta'))|^{3/2}}\lambda^4r'h_{2}(\lambda r')d\theta'dr'\\
    &=P.V.\int_{\mathds{R}_{+}\times[-\pi,\pi]}\bigg(\frac{\partial^2 }{\partial (r')^2}\frac{r-r'\cos(\theta')}{|r^2+(r')^2-2rr'\cos(\theta'))|^{3/2}}\bigg)\lambda^2r'h_{1}(\lambda r')d\theta'dr'
\end{align*}
and similarly, using integration by parts four times with respect to $r'$
\begin{align*}
    v_{\theta}(\lambda^6h_{3}(\lambda\cdot))(r,\alpha)&=P.V.\int_{\mathds{R}_{+}\times[-\pi,\pi]}\frac{r-r'\cos(\theta')}{|r^2+(r')^2-2rr'\cos(\theta'))|^{3/2}}\lambda^6r'h_{3}(\lambda r')d\theta'dr'\\
    &=P.V.\int_{\mathds{R}_{+}\times[-\pi,\pi]}\bigg(\frac{\partial^4 }{\partial (r')^4}\frac{r-r'\cos(\theta')}{|r^2+(r')^2-2rr'\cos(\theta'))|^{3/2}}\bigg)\lambda^2r'h_{1}(\lambda r')d\theta'dr'.
\end{align*}
Direct computation then gives us that

$$\text{lim}_{r'\rightarrow 0}\int_{0}^{2\pi}\frac{\partial^2 }{\partial (r')^2}\frac{r-r'\cos(\theta')}{|r^2+(r')^2-2rr'\cos(\theta'))|^{3/2}}d\theta'=\frac{3\pi}{ r^4}$$
$$\text{lim}_{r'\rightarrow 0}\int_{0}^{2\pi}\frac{\partial^4 }{\partial (r')^4}\frac{r-r'\cos(\theta')}{|r^2+(r')^2-2rr'\cos(\theta'))|^{3/2}}d\theta'=\frac{135\pi}{4 r^6}$$
and there is strong convergence in $C^{k}$ for $r\in(\frac12,\frac32)$ for any fixed $k$, so
\begin{equation}\label{h2}
    \text{lim}_{\lambda\rightarrow\infty}v_{\theta}(\lambda^4h_{2}(\lambda\cdot))(r,\alpha)'=\int_{\mathds{R}_{+}}\frac{3\pi}{ r^4}\lambda^2r'h_{1}(r')dr'=\frac{3\pi}{ r^4},
\end{equation}
\begin{equation}\label{h3}
    \text{lim}_{\lambda\rightarrow\infty}v_{\theta}(\lambda^6h_{3}(\lambda\cdot))(r,\alpha)=\int_{\mathds{R}_{+}}\frac{135\pi}{4 r^6}\lambda^2r'h_{1}(\lambda r')dr'=\frac{135\pi}{4 r^6},
\end{equation}
with, again, strong convergence in $C^{k}$ for $r\in(\frac12,\frac32)$ for any fixed $k$. Furthermore, if we now consider some $C^{\infty}$ radial function $ p(r)$ such that $p(r)=1$ if $r\geq 2$, $p(r)=0$ if $r\leq 1$, $1\geq p(r)\geq0$, and define $H_{c}$ as
$$\widehat{H_{c}(f(r))}=p(\frac{\hat{r}}{c})\hat{f}(\hat{r})$$ we have that for any $C^{\infty}$ function $H_{c}(f(r))$ tends strongly in $C^{k}$ to $f(r)$ as $c$ tends to $0$, and furthermore $H_{c}(f(r))$ is radial and with Fourier transform supported in $\hat{r}\in(c,\infty)$. Note also that $\hat{f}$ is smooth and therefore so is $p(\frac{\hat{r}}{c})\hat{f}(\hat{r})$.

Using this plus \eqref{h1},\eqref{h2} and \eqref{h3}, we have that we can find smooth  functions $g_{1}(r),g_{2}(r),g_{3}(r)$ with  $\hat{g_{i}}\in C^{\infty}$ supported in $\hat{r}\in(c,\infty)$ such that, for $r\in(\frac12,\frac32)$

$$\frac{\partial }{\partial r}\frac{v_{\theta}(g_{1})}{r}(r=1)=\frac{1}{r^4}+\epsilon,\frac{\partial^2 }{\partial r^2}\frac{v_{\theta}(g_{1})}{r}(r=1)=-\frac{4}{r^5}+\epsilon,\frac{\partial^3 }{\partial r^3}\frac{v_{\theta}(g_{1})}{r}(r=1)=\frac{20}{r^6}+\epsilon,$$
$$\frac{\partial }{\partial r}\frac{v_{\theta}(g_{2})}{r}(r=1)=\frac{1}{r^6}+\epsilon,\frac{\partial^2 }{\partial r^2}\frac{v_{\theta}(g_{2})}{r}(r=1)=-\frac{6}{r^7}+\epsilon,\frac{\partial^3 }{\partial r^3}\frac{v_{\theta}(g_{2})}{r}(r=1)=\frac{42}{r^8}+\epsilon,$$
$$\frac{\partial }{\partial r}\frac{v_{\theta}(g_{3})}{r}(r=1)=\frac{1}{r^8}+\epsilon,\frac{\partial^2 }{\partial r^2}\frac{v_{\theta}(g_{3})}{r}(r=1)=-\frac{8}{r^9}+\epsilon,\frac{\partial^3 }{\partial r^3}\frac{v_{\theta}(g_{3})}{r}(r=1)=\frac{72}{r^{10}}+\epsilon,$$
and since the vectors $(1,4,20),(1,6,42)$ and $(1,8,72)$ are independent, evaluating at $r=1$ and taking $\epsilon$ small finishes the proof.


\end{proof}

\section{The pseudo-solution}
We are now ready to define our pseudo-solution and obtain the necessary properties about it. First, taking 
\begin{equation}\label{nlambda}
    N^{\alpha}\ln{(N)}=\lambda^{2-\beta-\alpha}
\end{equation}
for $t\in[0,\lambda^{-2+\beta}(\ln{(N)})^3]=[0,(N\lambda)^{-\alpha}(\ln{(N)})^2]$, we define

\begin{equation}\label{radial+pert}
    \bar{w}_{N,\beta}(r,\theta,t):=\bar{g}(r,t)+\bar{w}_{pert}(r,\theta,t)
\end{equation}
where

$$\bar{g}(r,t):=\lambda\frac{g(r\lambda,t\lambda^{\alpha})}{\lambda^{\beta}}$$
$$\frac{\partial g(r,t)}{\partial t}=-\Lambda^{\alpha}g(r,t),$$
$$g(r,0)=g(r),$$
\begin{align}\label{wpertdef}
    \bar{w}_{pert}(r,\theta,t):=f(\lambda r) \lambda \frac{\cos(N(\theta+\Theta(r,t))-\int_{0}^{t}\frac{\partial \bar{g}(r,s)}{\partial r}\frac{C_{0}}{{|1+r^2(\partial_{r}\Theta(r,s))^2|^{1/2}}}ds) }{N^{\beta}\lambda^{\beta}}e^{-G(r,t)}
\end{align}
with
$$\Theta(r,t):=-\big(K\frac{v_{\theta}(g(\lambda r,0))}{\lambda r}+\int_{0}^{t}\frac{v_{\theta}(\bar{g}(r,s))}{r}ds\big)$$
$$G(r,t):=K_{\alpha}^{-1}N^{\alpha}\int_{0}^{t}\Big(\frac{1}{r^2}+(\frac{\partial}{\partial r} \Theta(r,s))^2\Big)^{\frac{\alpha}{2}}ds,$$
where $C_{0}$ is the constant from lemma \ref{vr} and $K$, $g(r)$ and $f(r)$ will be fixed later.
Note that, with this choice of $\bar{w}_{pert}$, by taking a derivative in time we can check that 
\begin{align}\label{pseudowpert}
    \p_{t}\bar{w}_{pert}(r,\theta,t)+v_{\theta}(\bar{g}(r,t))\frac{1}{r} \frac{\partial}{\partial \theta}\bar{w}_{pert}+\bar{v}_{r}(\bar{w}_{pert})\frac{\partial}{\partial r}\bar{g}(r,t)+\bar{\Lambda}^{\alpha}\bar{w}_{pert}=0
\end{align}
with
$$\bar{\Lambda}^{\alpha}\bar{w}_{pert}:=K_{\alpha}^{-1}\bar{w}_{pert}|\Big(\frac{N}{r}\Big)^2+\Big(N\frac{\partial\Theta(r,t)}{\partial r}\Big)^2|^{\alpha/2}$$
and
$$\bar{v}_{r}(h_{1}(r)\cos(N\theta+h_{2}(r)):=-C_{0}\frac{h_{1}(r)\sin(N\theta+Nh_{2}(r))}{|1+r^2h_{2}'(r)^2|^{1/2}}.$$ This motivates our choice of $\bar{w}_{pert}$, since it closely resembles the evolution one would obtain when perturbing the solution to SQG $\bar{g}(r,t)$ with some perturbation $w_{pert}$, which would give

$$\p_{t}w_{pert}(r,\theta,t)+v_{\theta}(\bar{g}(r,t))\frac{1}{r} \frac{\partial}{\partial \theta}w_{pert}+v_{r}(w_{pert})\frac{\partial}{\partial r}\bar{g}(r,t)+v(w_{pert})\cdot\nabla w_{pert}+\Lambda^{\alpha}\bar{w}_{pert}=0.$$

Note also that, even though the $\bar{w}_{pert}$ fulfills equation \eqref{pseudowpert} independently of the choice of $\Theta(r,0)$, we choose 
$$\Theta(r,0)=-K\frac{v_{\theta}(g(\lambda r,0))}{\lambda r}$$
so that the derivative in $r$ of $w_{pert}$ dominates the derivative in $\theta$. This will be relevant when studying the behaviour of $G(r,t)$.

Now, to choose our $K$, $g(r)$ and $f(r)$ we start by choosing a smooth radial function $g(r)$ with $\text{supp}(\hat{g})\subset \{\hat{r}\in(c,\infty)\}$ for some small $c$ such that

$$\frac{\partial}{\partial r}\frac{v_{\theta}(g(r))}{r}(r=1)=1,$$
$$\frac{\partial^2}{\partial r^2}\frac{v_{\theta}(g(r))}{r}(r=1)=0,$$
$$\frac{\partial^3}{\partial r^3}\frac{v_{\theta}(g(r))}{r}(r=1)=1,$$
which exists thanks to lemma \ref{initialvel}. Note that this means that there is a small $\frac12>\tilde{\epsilon}>0$ such that $r_{0}\in[1-\tilde{\epsilon},1+\tilde{\epsilon}]$,  implies that 

\begin{equation} \label{minimum}
    \big(\frac{\partial }{\partial r}\frac{v_{\theta}(g(r))}{r}\big)(r=r_{0})-\big(\frac{\partial }{\partial r}\frac{v_{\theta}(g(r))}{r}\big)(r=1)\geq \frac{1}{10}(1-r_{0})^2.
\end{equation}

We would now like to choose $K$ so that, if $N$ is big enough, $G(r,t)$ is such that, for $t\in[0,(\lambda N)^{-\alpha}(\ln(N))^2], r \in[\frac{1-\tilde{\epsilon}}{\lambda},\frac{2-\tilde{\epsilon}}{2\lambda}]\cup[\frac{2+\tilde{\epsilon}}{2\lambda},\frac{1+\tilde{\epsilon}}{\lambda}]$

\begin{equation}\label{maximocentro}
    G(\frac{1}{\lambda},t)\leq G(r,t).
\end{equation}
For this, we note that it is enough to show that, for $t\in[0,(\lambda N)^{-\alpha}(\ln(N))^2]$, $r\in[\frac{1-\tilde{\epsilon}}{\lambda},\frac{2-\tilde{\epsilon}}{2\lambda}]\cup[\frac{2+\tilde{\epsilon}}{2\lambda},\frac{1+\tilde{\epsilon}}{\lambda}]$,

$$\lambda^2+(\frac{\partial}{\partial r} \Theta(r,t)(r=\frac{1}{\lambda}))^2\leq \frac{1}{r^2}+(\frac{\partial}{\partial r} \Theta(r,t))^2.$$

But,  using the definition of $\bar{g}$

\begin{align}\label{vctt}
    \int_{0}^{t}\frac{\partial }{\partial r}\frac{v_{\theta}(\bar{g}(\cdot,\tau))( r)}{r}d\tau=t\lambda^{3-\beta}\frac{\partial }{\partial (r\lambda)}\frac{v_{\theta}(g(r\lambda,0))}{\lambda r}+\int_{0}^{t}\frac{\partial }{\partial r}\frac{v_{\theta}(\bar{g}(\cdot,\tau)-\bar{g}(\cdot,0))( r)}{r}d\tau
\end{align}
and using $t\in[0,(\lambda N)^{-\alpha}(\ln(N))^2], r \in[\frac{1-\tilde{\epsilon}}{\lambda},\frac{1+\tilde{\epsilon}}{\lambda}]$ we have,
\begin{align*}
    &|\int_{0}^{t}\frac{\partial }{\partial r}\frac{v_{\theta}(\bar{g}(\cdot,\tau)-\bar{g}(\cdot,0))( r)}{r}d\tau|=\lambda^{3-\beta}|\int_{0}^{t}\frac{\partial }{\partial (\lambda r)}\frac{v_{\theta}(g(\lambda \cdot,\lambda^{\alpha}\tau)-g(\lambda \cdot,0))( r)}{\lambda r}d\tau|\\
    &\leq C\lambda^{3-\beta} \lambda^{\alpha} t^2\leq Ct\lambda^{3-\beta}  N^{-\alpha}(\ln(N))^2
\end{align*}
so, for big $N$, $r\in[\frac{1-\tilde{\epsilon}}{\lambda},\frac{2-\tilde{\epsilon}}{2\lambda}]\cup[\frac{2+\tilde{\epsilon}}{2\lambda},\frac{1+\tilde{\epsilon}}{\lambda}]$, using the definition of $\Theta(r,t)$, \eqref{vctt} and \eqref{minimum} we have

$$\frac{1}{r^2}+(\frac{\partial}{\partial r} \Theta(r,t))^2\geq \lambda^2(1-\tilde{\epsilon})^{-2}+((K\lambda+t\lambda^{3-\beta})(1+\frac{\tilde{\epsilon}^2}{40})-Ct\lambda^{3-\beta}  N^{-\alpha}(\ln(N))^2)^2$$
and also
$$\lambda^2+(\frac{\partial}{\partial r} \Theta(r,t))^2(r=\frac{1}{\lambda})\leq \lambda^2+(K\lambda+t\lambda^{3-\beta}+Ct\lambda^{(3-\beta)} N^{-\alpha}(\ln(N))^2)^2$$
so by taking $K,N$ big (and therefore $\lambda$ big), gives us that \eqref{maximocentro} is fulfilled.

We now take $f(r)$ a smooth function with support in $r\in [1-\tilde{\epsilon},1+\tilde{\epsilon}]$ the small interval fulfilling \eqref{maximocentro}, and such that $f(r)=1$ if $r\in[1-\frac{\tilde{\epsilon}}{2},1+\frac{\tilde{\epsilon}}{2}]$, $f(r)\leq 1$. Note that, since the minimum of $G(r,t)1_{r\in[\frac{1}{\lambda}(1-\tilde{\epsilon}),\frac{1}{\lambda}(1+\tilde{\epsilon})]}$ is in the interval $r\in[\frac{1}{\lambda}(1-\frac{\tilde{\epsilon}}{2}),\frac{1}{\lambda}(1+\frac{\tilde{\epsilon}}{2})]$, this ensures that

\begin{equation}\label{maximoproducto}
    ||\bar{w}_{pert}||_{L^\infty}= \frac{||f(\lambda r)||_{L^{\infty}}}{N^{\beta}\lambda^{\beta-1}}||1_{\text{supp} f(\lambda r)}e^{-G(r,t)}||_{L^{\infty}}.
\end{equation}

Note that, by simply adding the evolution equations for $\bar{g}(r,t)$, $\bar{w}_{pert}$ \eqref{wpertdef}, we obtain that the pseudo-solution fulfills the evolution equation

$$\frac{\partial \bar{w}_{N,\beta}}{\partial t}+v(\bar{g}(r,t))\cdot\nabla(\bar{w}_{N,\beta})+\bar{v}_{r}(\bar{w}_{pert})\frac{\partial}{\partial r} \bar{g}(r,t)+\Lambda^{\alpha}(\bar{g}(r,t))+\bar{\Lambda}^{\alpha}(\bar{w}_{pert})=0$$
so that
\begin{equation}\label{pseudoevolucion}
    \frac{\partial \bar{w}_{N,\beta}}{\partial t}+v(\bar{w}_{N,\beta})\cdot\nabla(\bar{w}_{N,\beta})+\Lambda^{\alpha}(\bar{w}_{N,\beta})+F_{N,\beta}(x,t)=0
\end{equation}
with 
\begin{equation}\label{F123}
    F_{N,\beta}(x,t):=F_{1}(x,t)+F_{2}(x,t)+F_{3}(x,t)
\end{equation}

$$F_{1}(x,t)=(\bar{\Lambda}^{\alpha}-\Lambda^{\alpha})(\bar{w}_{pert}),$$
$$F_{2}(x,t)=-v(\bar{w}_{pert})\cdot\nabla(\bar{w}_{pert})$$
$$F_{3}(x,t)=(\bar{v}(\bar{w}_{pert})-v(\bar{w}_{pert}))\cdot \nabla \bar{g}(r,t).$$

Next we need to show that $F_{N,\beta}$ is small in suitable Sobolev spaces.

\subsection{Bounds on the error term $F_{N,\beta}$}



\begin{lemma}
For any given $\epsilon>0$, there is $N_{0}$ such that if $N\geq N_{0}$, then for $F_{N,\beta}$ given by \eqref{pseudoevolucion} and $s\in[0,2]$, we have that

$$||F_{N,\beta}||_{H^{s}}\leq C_{\epsilon} \frac{N^{\epsilon} (\lambda N)^{s+\alpha}}{N^{\beta+1}\lambda^{\beta}}$$
with $N,\lambda$ as in \eqref{nlambda}. 
\end{lemma}

\begin{proof}
    To prove this, we will just show that
    $$||F_{i}||_{H^{s}}\leq C_{\epsilon} \frac{N^{\epsilon} (\lambda N)^{s+\alpha}}{N^{\beta+1}\lambda^{\beta}}.$$
    for $i=1,2,3$, with $F_{i}$ defined as in (\ref{F123}). We first focus on $F_{1}(x,t)$. To bound the $H^s$ norms of this function, we will use corollary \ref{+alpha}, so for that we need to check if $\bar{w}_{pert}$ fulfils the hypothesis of the lemma. For this we note that 

    $$\bar{w}_{pert}(\frac{r}{\lambda},\theta,t)=f( r) \lambda \frac{\cos(N(\theta+\Theta(\frac{r}{\lambda},t))-\int_{0}^{t}\lambda^{2-\beta} \frac{\partial g(r,\lambda^{\alpha}s)}{\partial r}\frac{C_{0}}{{|1+(r\partial_{r}\Theta(\frac{r}{\lambda},s))^2|^{1/2}}}ds) }{N^{\beta}\lambda^{\beta}}e^{-G(\frac{r}{\lambda},t)}$$
    $$G(\frac{r}{\lambda},t)=K_{\alpha}^{-1}(N\lambda)^{\alpha}\int_{0}^{t}\Big(\frac{1}{r^2}+(K\frac{\partial}{\partial r}\frac{v_{\theta}(g( r,0))}{ r}+\int_{0}^{s}\lambda^{2-\beta}\frac{\partial }{\partial r}\frac{v_{\theta}(g(\cdot,\lambda^{\alpha}\tau))( r)}{r}d\tau)^2\Big)^{\frac{\alpha}{2}}ds,$$
    $$\Theta(\frac{r}{\lambda},t)=-\Big(K \frac{v_{\theta}(g(r,0))}{r}+\lambda^{2-\beta}\int_{0}^{t}\frac{v_{\theta}(g(r,\lambda^{\alpha}s))}{r}ds\Big)$$

    Since $f(r)e^{-G(\frac{r}{\lambda},t)}$ has support in $r\in(\frac{1}{2},\frac{3}{2})$, if we show that 
    \begin{equation}\label{c5linfty}
        ||f(r)e^{-G(\frac{r}{\lambda},t)}||_{C^{5}}\leq ||f(r)e^{-G(\frac{r}{\lambda},t)}||_{L^{\infty}}(\ln(N))^{P},
    \end{equation}
        
    $$||K\frac{v_{\theta}(g( r,0))}{ r}+\int_{0}^{t}\frac{\lambda^{2-\beta} v_{\theta}(g(r,\lambda^{\alpha}s))}{r}ds||_{C^5}\leq (\ln(N))^{P},$$

    $$||\int_{0}^{t}\lambda^{2-\beta} \frac{\partial g(r,\lambda^{\alpha}s)}{\partial r}\frac{C_{0}}{{|1+(r\partial_{r}\Theta(\frac{r}{\lambda},s))^2|^{1/2}}}ds)||_{C^{5}}\leq (\ln(N))^{P}$$
    we can apply corollary \ref{+alpha}.

    For \eqref{c5linfty} we note that, for a function of the form $\tilde{h}(x)=h_{1}(x)e^{h_{2}(x)}$, we have the bound 
    $$||\tilde{h}(x)||_{C^{i}}\leq C_{i}||1_{\text{supp}(h_{1}(x))}e^{h_{2}(x)}||_{L^{\infty}}||h_{1}||_{C^{i}}(1+||h_{2}(x)||_{C^{i}})^{i}$$
    but,  using also that $f(r)$ is a fixed $C^{\infty}$ function and \eqref{maximoproducto}
\begin{align*}
    &\frac{||f(r)e^{-G(\frac{r}{\lambda},t)}||_{C^{5}}}{||f(r)e^{-G(\frac{r}{\lambda},t)}||_{L^{\infty}}}\\
    &\leq C\frac{||1_{\text{supp}f(r)}e^{-G(\frac{r}{\lambda})(x)}||_{L^{\infty}}||f(r)||_{C^{5}}(1+||G(\frac{r}{\lambda},t)||_{C^{5}(1_{\text{supp} f(r)})})^5}{||f(r)||_{L^{\infty}}||1_{\text{supp} f(r)}e^{-G(\frac{r}{\lambda},t)}||_{L^{\infty}}}\\
    &\leq C(1+||G(\frac{r}{\lambda},t)||_{C^{5}(1_{\text{supp} f(r)})})^5
\end{align*}
    so it is enough to obtain bounds for $||G(\frac{r}{\lambda},t)||_{C^{5}(1_{\text{supp} f(r)})}$. But, for $t\in[0,(\lambda N)^{-\alpha}(\ln(N))^2]$

    \begin{align*}
        &||G(\frac{r}{\lambda},t)||_{C^{5}(1_{\text{supp} f(r)})}\\
        &=||K_{\alpha}^{-1}(N\lambda)^{\alpha}\int_{0}^{t}\Big(\frac{1}{r^2}+(K\frac{v_{\theta}(g( r,0))}{ r}+\int_{0}^{s}\lambda^{2-\beta}\frac{\partial }{\partial r}v_{\theta}(g(\cdot,\lambda^{\alpha}\tau))( r)d\tau)^2\Big)^{\frac{\alpha}{2}}||_{C^{5}(1_{\text{supp} f(r)})}\\
        &\leq C(\ln(N))^2\times\\
        &\text{sup}_{s\in[0,(N\lambda)^{-\alpha}(\ln(N))^2]}||\Big(\frac{1}{r^2}+(K\frac{v_{\theta}(g( r,0))}{ r}+\int_{0}^{s}\lambda^{2-\beta}\frac{\partial }{\partial r}v_{\theta}(g(\cdot,\lambda^{\alpha}\tau))( r)d\tau)^2\Big)^{\frac{\alpha}{2}}||_{C^{5}(1_{\text{supp} f(r)}))}\\
        &\leq C(\ln(N))^2\times\\
        &\text{sup}_{s\in[0,(N\lambda)^{-\alpha}(\ln(N))^2]}(1+||\frac{1}{r^2}+(K\frac{v_{\theta}(g( r,0))}{ r}+\int_{0}^{s}\lambda^{2-\beta}\frac{\partial }{\partial r}v_{\theta}(g(\cdot,\lambda^{\alpha}\tau))( r)d\tau)^2||_{C^{5}(1_{\text{supp} f(r)}))})^{5}\\
        &\leq C (\ln(N))^{32}
    \end{align*}
where we used that  $\frac{1}{r^2}+(K\frac{v_{\theta}(g( r,0))}{ r}+\int_{0}^{s}\lambda^{2-\beta}\frac{\partial }{\partial r}v_{\theta}(g(\cdot,\lambda^{\alpha}\tau))( r)d\tau)^2>\frac{1}{4}$ for $r\in \text{supp}(f(r))$ on the fifth line.  

On the other hand, for $t\in[0,(\lambda N)^{-\alpha}(\ln(N))^2]$ we have

$$||K\frac{v_{\theta}(g( r,0))}{ r}+\int_{0}^{t}\frac{\lambda^{2-\beta} v_{\theta}(g(r,\lambda^{\alpha}s))}{r}ds||_{C^{5}}\leq C(\ln(N))^3$$
\begin{align*}
    &||\int_{0}^{t}\lambda^{2-\beta} \frac{\partial g(r,\lambda^{\alpha}s)}{\partial r}\frac{C_{0}}{{|1+r^2(\partial_{r}\Theta(\frac{r}{\lambda},s))^2|^{1/2}}}ds||_{C^{5}}\\
    &\leq C(\ln(N))^3\text{sup}_{s\in[0,(\lambda N)^{-\alpha}(\ln(N))^2]}||\frac{C_{0}}{{|1+r^2(\partial_{r}\Theta(\frac{r}{\lambda},s))^2|^{1/2}}}||_{C^{5}}\leq C (\ln(N))^{33}
\end{align*}
$$$$
so we can apply corollary \ref{+alpha} and obtain that

$$||(\Lambda^{\alpha}-\bar{\Lambda}^{\alpha})\bar{w}_{pert}(r,\theta)||_{H^{s}}\leq C\lambda^{-1+s+\alpha}N^{-1+\alpha+\epsilon+s}||\frac{f(\lambda r)}{N^{\beta}\lambda^{\beta-1}} e^{-G(r,t)}||_{L^{\infty}}\leq C\frac{\lambda^{s+\alpha}N^{\epsilon+\alpha+s}}{\lambda^{\beta}N^{\beta+1}}.$$

Note also that, all the bounds we have obtained also give us, for any $\epsilon>0$, $s\in[0,5]$ and $N$ big

\begin{equation}\label{sobolevpert}
    ||\bar{w}_{pert}(r,\theta)||_{H^{s}}\leq C\frac{\lambda^{s}N^{\epsilon+s}}{\lambda^{\beta}N^{\beta}},||\bar{w}_{pert}(r,\theta)||_{C^{s}}\leq C\frac{\lambda^{s}N^{\epsilon+s}}{\lambda^{\beta-1}N^{\beta}}.
\end{equation}

Next, for $F_{2}(x,t)$, we note that

$$||F_{2}(x,t)||_{H^{s}}\leq ||(\bar{v}-v)(\bar{w}_{pert})\cdot\nabla(\bar{w}_{pert})||_{H^{s}}+||\bar{v}(\bar{w}_{pert})\cdot\nabla(\bar{w}_{pert})||_{H^{s}}.$$

But, since we already checked the hypothesis for lemma \ref{vcuad}, we have, for $s=0,2 $ and any $\epsilon>0$

$$||(\bar{v}-v)(\bar{w}_{pert})\cdot\nabla(\bar{w}_{pert})||_{H^{s}}\leq C\sum_{i=0}^{s}||(\bar{v}-v)(\bar{w}_{pert})||_{H^{i}}||\nabla(\bar{w}_{pert})||_{C^{s-i}}\leq C\frac{N^{2\epsilon}(\lambda N)^{s}}{N^{\beta+1}\lambda^{\beta}} \lambda^{2-\beta}$$
$$=C\frac{N^{2\epsilon}(\lambda N)^{s+\alpha}}{N^{\beta+1}\lambda^{\beta}}\ln{(N)}\leq C\frac{N^{3\epsilon}(\lambda N)^{s+\alpha}}{N^{\beta+1}\lambda^{\beta}}$$
and interpolation gives the bound for $s\in(0,2)$. On the other hand we have

$$\bar{v}(\bar{w}_{pert})\cdot\nabla(\bar{w}_{pert})=K_{1}\big(\frac{\partial }{\partial x_{1}}\frac{1}{|(\frac{N}{r})^2+(N\partial_{r}\Theta(r,t))^2|^{1/2}}\big) \bar{w}_{pert}(r,\theta,t)\frac{\partial }{\partial x_{2}}\bar{w}_{pert}(r,\theta,t)$$
$$-K_{1}\big(\frac{\partial }{\partial x_{2}}\frac{1}{|(\frac{N}{r})^2+(N\partial_{r}\Theta(r,t))^2|^{1/2}}\big) \bar{w}_{pert}(r,\theta,t)\frac{\partial }{\partial x_{1}}\bar{w}_{pert}(r,\theta,t)$$
and therefore, for $s\in[0,2]$

\begin{align*}
    &||\bar{v}(\bar{w}_{pert})\cdot\nabla(\bar{w}_{pert})||_{H^{s}}\\
    &\leq C\sum_{i=0}^{s}||\nabla\big(\frac{1}{|(\frac{N}{r})^2+(N\partial_{r}\Theta(r,t))^2|^{1/2}}\big)\bar{w}_{pert}(r,\theta)||_{C^{i}}||\bar{w}_{pert}(r,\theta)||_{H^{s+1-i}}\\
    &\leq C \sum_{i=0}^{s}\lambda^{i+1-\beta}N^{i-\beta-1+2\epsilon}(\lambda N)^{s+1-i-\beta+\epsilon}=C\lambda^{1-\beta}N^{-\beta-1+2\epsilon}(\lambda N)^{s+1-\beta+\epsilon}\\
    &=C\frac{(\lambda N)^{s} N^{3\epsilon}}{\lambda^{\beta}N^{\beta+1}}\lambda^{2-\beta} N^{1-\beta}\leq C\frac{(\lambda N)^{s+\alpha} N^{3\epsilon}}{\lambda^{\beta}N^{\beta+1}}.
\end{align*}

Finally, for $F_{3}(x,t)$, we just have, for any $s\in[0,2]$, $\epsilon>0$, for $N$ big enough

\begin{align*}
    &||(\bar{v}-v)(\bar{w}_{pert})\cdot \nabla \bar{g}(r,t)||_{H^{s}}\leq C\sum_{i=0}^{s}||(\bar{v}-v)(\bar{w}_{pert})||_{H^{i}}||\bar{g}||_{C^{s+1-i}}\\
    &\leq C \frac{N^{\epsilon} (N\lambda)^{i}}{\lambda^{\beta}N^{\beta+1}}\lambda^{s+2-i-\beta}\leq C \frac{N^{\epsilon} (N\lambda)^{s}}{\lambda^{\beta}N^{\beta+1}}\lambda^{2-\beta}\leq C \frac{N^{2\epsilon} (N\lambda)^{s+\alpha}}{\lambda^{\beta}N^{\beta+1}}
\end{align*}
and this finishes the proof.

\end{proof}

\subsection{Using the pseudo-solution to control the solution}
For the pseudo-solution $\bar{w}_{N,\beta}$ to be a useful tool, we need to show that, if we define $w_{N,\beta}$, the solution to \eqref{SQG} with the same initial conditions as $\bar{w}_{N,\beta}$, then $w_{N,\beta}\approx \bar{w}_{N,\beta}$. This sub-section will be devoted to show this.

\begin{lemma}\label{shortcontrol}
There is $N_{0},\epsilon_{0},\delta >0$ such that if $N\geq N_{0}$, $0<\epsilon\leq \epsilon_{0}$ then  for any $t\in[0,(N\lambda)^{-\alpha}(\ln(N))^2]$ we have that

\begin{equation}\label{controll2}
    ||w_{N,\beta}(x,t)-\bar{w}_{N,\beta}(x,t)||_{L^2}\leq C_{\epsilon}\frac{N^{\epsilon}}{N^{\beta+1}\lambda^{\beta}},
\end{equation}

\begin{equation}\label{controlhs}
    ||w_{N,\beta}(x,t)-\bar{w}_{N,\beta}(x,t)||_{H^{2-\alpha+\delta}}\leq 1,
\end{equation}
with $\bar{w}_{N,\beta}$ as in \eqref{radial+pert} and $w_{N,\beta}$ a solution to \eqref{SQG} with the same initial conditions as $\bar{w}_{N,\beta}$.

\end{lemma}

\begin{proof}
    First, we recall that, due the scalling properties of $C^{k}$ and the definition of $\bar{g}(r,t)$, we have that
    \begin{equation}\label{boundgbar}
        ||\bar{g}(r,t)||_{C^{k}}\leq C \lambda^{k+1-\beta}.
    \end{equation}
    Now, we note that, since \eqref{SQG} is locally well-posed in $H^{2-\alpha+\delta}$, using continuity of \ $||w_{N,\beta}(x,t)-\bar{w}_{N,\beta}(x,t)||_{H^{2-\alpha+\delta}}$ we know that \eqref{controlhs} will hold for at least some short time period $[0,T_{crit}]$. Note also that, for this times, $w_{N,\beta}(x,t)\in C^{\infty}$, since the initial conditions are smooth and we are controlling a sub-critical norm. We start by showing that, for $t\in[0,T_{crit}]\cap[0,(N\lambda)^{-\alpha}(\ln(N))^2]$, \eqref{controll2} holds.

    For this, we define $W:=w_{N,\beta}(x,t)-\bar{w}_{N,\beta}(x,t)$ and note that the evolution equation for $W$ is

    $$\frac{\partial W}{\partial t}+v(W)\cdot\nabla (W+\bar{w}_{N,\beta}(x,t))+v(\bar{w}_{N,\beta}(x,t))\cdot\nabla W+\Lambda^{\alpha} W-F(x,t)=0$$
    so that, after using incompressibility,

    \begin{align}\label{L2evol}
        &\frac{\partial}{\partial t}||W||^2_{L^2}=-\Big(2\int_{\mathds{R}^2}W(v(W)\cdot\nabla\bar{w}_{N,\beta}-F(x,t))dx+||W||^2_{\dot{H}^{\frac{\alpha}{2}}}\Big)\\\nonumber
        &\leq -\Big(2\int_{\mathds{R}^2}W(v(W)\cdot\nabla\bar{g}(r,t))-F(x,t))dx\Big)+C\lambda^{2-\beta}N^{\beta-1}\ln{(N)}||W||^2_{L^2} \\\nonumber
    \end{align}
    where we used that $||\bar{w}_{pert}(r,\theta,t)||_{C^1}\leq C\lambda^{2-\beta}N^{\beta-1}(\ln{(N))^3}$.
  Now,  we consider $\tilde{f}(r)$ a smooth function fulfilling $0\leq \tilde{f}(r)\leq 1$, $\tilde{f}(r)=1$ if $r\geq 2$, $\tilde{f}(r)=0$ if $r\leq 1$ and we define, for $k\in\mathds{Z}$
  $$f_{k}(r)=\tilde{f}(\lambda 2^{k}r)-\tilde{f}(\lambda 2^{k-1}r),$$
  which has support in $r\in[2^{-k}\lambda^{-1},2^{-k}4\lambda^{-1}]$.
  Note that $\sum_{k=-\infty}^{\infty}f_{k}=1$. We also define
  $$f_{\infty}:=\sum^{\infty}_{k=\lceil \log_{2}(\lambda^{2-\beta})\rceil}f_{k}(r)=1-\tilde{f}(2^{\lceil \log_{2}(\lambda^{2-\beta})\rceil}\lambda r),\ f_{-\infty}:=\sum_{k=-\infty}^{\lfloor \log_{2}(-\lambda)\rfloor}f_{k}(r)=\tilde{f}(2^{\lfloor -\log_{2}(\lambda)\rfloor}\lambda r).
  $$

First, since $\bar{g}(r,t)$ is $C^2$, which in particular implies $\partial_{r}\bar{g}(0,t)=0$ (since $-\p_{x_{1}}g(x=0)=\p_{r}g(r=0)=\p_{x_{1}}g(x=0)$), and using that, if $D^{j}$ is a generic derivative of order $j$, $||D^{j}f(Kx)||_{L^{\infty}}=K^{j}||D^{j}f(x)||_{L^{\infty}}$, we get that
\begin{align*}
    &||f_{\infty}(r)\p_{x_{i}}\bar{g}(r,t)||_{L^{\infty}}\leq ||\p_{r}\bar{g}(r,t)||_{L^{\infty}(\text{supp}f_{\infty})}\\
    &\leq||\p_{r}\bar{g}(r,t)||_{L^{\infty}(B_{8\lambda^{-3+\beta}}(0))}\leq C\lambda^{-3+\beta}||\p_{r}\p_{r}\bar{g}(r,t)||_{L^{\infty}}\leq C.
\end{align*}
and thus
$$||\int_{\mathds{R}^2}f_{\infty}(r)W(v(W)\cdot\nabla\bar{g}(r,t))dx||_{L^2}\leq C||W||_{L^2}||v(W)||_{L^2}\leq C||W||^2_{L^2}.$$
Next we want to bound the contribution we obtain when multiplying by $f_{-\infty}$. For this we use lemma \ref{radialdecay} to obtain that
$$|\p_{r}\bar{g}(r,t)(r=r_{0})|=|\lambda^{1-\beta}\p_{r}{g(\lambda r,t)}(r=r_{0})|=|\lambda^{2-\beta}(\p_{r}g)(\lambda r_{0},t)|\leq C\lambda^{2-\beta}\frac{1}{r_{0}\lambda}$$
and thus
\begin{align*}
    &||f_{-\infty}(r)\p_{x_{i}}\bar{g}(r,t)||_{L^{\infty}}\leq ||\p_{r}\bar{g}(r,t)||_{L^{\infty}(\text{supp}f_{-\infty})}\leq||\p_{r}\bar{g}(r,t)||_{L^{\infty}(r\geq \frac12)}\leq C\lambda^{2-\beta}\frac{2}{\lambda}\leq C,
\end{align*}
so
$$||\int_{\mathds{R}^2}f_{-\infty}(r)W(v(W)\cdot\nabla\bar{g}(r,t))dx||_{L^2}\leq C||W||_{L^2}||v(W)||_{L^2}\leq C||W||^2_{L^2}.$$
Next, we want to bound the contribution when multiplying by some generic $f_{k}$. We first note that we can decompose $W(r,\theta,t)$ as
$$W(r,\theta,t)=\sum_{j=0}^{\infty}W_{j}=\sum_{i=0}^{\infty}p_{N}(r)\cos(jN\theta+h_{N}(r))$$
with $p_{N},h_{N}\in C^1$. Note that, by using the orthogonality of the different modes $\cos(jN\theta)$, we have
\begin{align*}
    &\int_{\mathds{R}^2}f_{k}(r)W(v(W)\cdot\nabla\bar{g}(r,t))dx=\sum_{j=0}^{\infty}\int_{\mathds{R}^2}f_{k}(r)W_{j}(v(W_{j})\cdot\nabla\bar{g}(r,t))dx\\
    &=\sum_{j=1}^{\infty}\int_{\mathds{R}^2}f_{k}(r)W_{j}(v(W_{j})\cdot\nabla\bar{g}(r,t))dx.
\end{align*}

Furthermore, we can then use corollary \ref{commutardorlambda} with $\omega(r,\theta)=W_{j}$, $Nj=N$, $g(r)\sin(\theta+i\frac{\pi}{2})=f_{k}(r)\p_{x_{i}}\bar{g}(r,t)$, $\lambda^{-1}=2^{-k}\lambda^{-1} $  plus the parity of the operator $v_{i}$ for $i=1,2$ to get, for $k\geq 0$

  \begin{align*}
      &||\int_{\mathds{R}^2}W_{j}v_{i}(W_{j})f_{k}(r)\frac{\partial}{\partial x_{i}} \bar{g}(r,t) dx||_{L^2}\\
      &=||\frac{1}{2}\int_{\mathds{R}^2}W_{j}[v_{i}(W_{j})f_{k}(r)\frac{\partial}{\partial x_{i}} \bar{g}(r,t)-v_{i}(W_{j} f_{k}(r)\frac{\partial}{\partial x_{i}}\bar{g}(r,t))]dx||_{L^2}\\
      &\leq C_{\epsilon}N^{-1+\epsilon}||f_{k}(\frac{r}{2^{k}\lambda})(\frac{\partial}{\partial x_{i}} \bar{g})(\frac{r}{2^{k}\lambda},t)||_{C^1}||W_{j}||^2_{L^2}\leq CN^{-1+\epsilon}\lambda^{2-\beta}||f_{k}(\frac{r}{2^{k}\lambda})||_{C^1}||(\frac{\partial}{\partial x_{i}} g)(\frac{r}{2^{k}},\lambda^{\alpha}t)||_{C^1}||W_{j}||_{L^2}^2\\
      &\leq CN^{-1+\epsilon}\lambda^{2-\beta}||(\frac{\partial}{\partial x_{i}} g)(\frac{r}{2^{k}},\lambda^{\alpha}t)||_{C^1}||W_{j}||_{L^2}^2\leq CN^{-1+\epsilon}\lambda^{2-\beta}||W_{j}||_{L^2}^2\leq C\lambda^{\alpha}\ln(N)N^{-1+\alpha+\epsilon}||W_{j}||^2.
  \end{align*}

 Finally, for $k<0$, We again use corollary \ref{commutardorlambda} with the same choices, obtaining
 \begin{align*}
      &||\int_{\mathds{R}^2}W_{j}v_{i}(W_{j})f_{k}(r)\frac{\partial}{\partial x_{i}} \bar{g}(r,t) dx||_{L^2}\leq CN^{-1+\epsilon}\lambda^{2-\beta}||(\frac{\partial}{\partial x_{i}} g)(2^{-k}r,t)||_{C^1(r\in[1,4]))}||W_{j}||_{L^2}^2,\\
      &\leq CN^{-1+\epsilon}\lambda^{2-\beta}2^{-k}||(\frac{\partial}{\partial x_{i}} g)(r,t)||_{C^1(r\in[2^{-k},2^{-k}4]))}||W_{j}||_{L^2}^2
  \end{align*}
  and we can use lemma \ref{radialdecay} to obtain
  \begin{align*}
      &||\int_{\mathds{R}^2}W_{j}v_{i}(W_{j})f_{k}(r)\frac{\partial}{\partial x_{i}} \bar{g}(r,t) dx||_{L^2}\leq CN^{-1+\epsilon}\lambda^{2-\beta}||W_{j}||^2.
  \end{align*}
 Adding over all the $j$ we get, for any $k$
 $$||\int_{\mathds{R}^2}Wv_{i}(W)f_{k}(r)\frac{\partial}{\partial x_{i}} \bar{g}(r,t) dx||_{L^2}\leq CN^{-1+\epsilon}\lambda^{2-\beta}||W_{j}||^2.$$
 To finish, we just combine the different inequalities obtained
 \begin{align*}
     &||\int_{\mathds{R}^2}Wv_{i}(W)\frac{\partial}{\partial x_{i}} \bar{g}(r,t) dx||_{L^2}=||\int_{\mathds{R}^2}\sum_{k=-\infty}^{\infty}f_{k}Wv_{i}(W)\frac{\partial}{\partial x_{i}} \bar{g}(r,t) dx||_{L^2}\\
     &\leq ||\int_{\mathds{R}^2}\sum_{k=\lfloor \log_{2}(-\lambda)\rfloor}^{k=\lceil \log_{2}(\lambda^{2-\beta})\rceil}f_{k}Wv_{i}(W)\frac{\partial}{\partial x_{i}} \bar{g}(r,t) dx||_{L^2}+C||W||_{L^2}\\
     &\leq CN^{-1+\epsilon}\ln(\lambda)\lambda^{2-\beta}||W||_{L^2}
 \end{align*}
Combining all this we get that

    \begin{align*}
        &\frac{\partial}{\partial t}||W||^2_{L^2}=-2\Big(\int_{\mathds{R}^2}W(v(W)\cdot\nabla\bar{w}_{N,\beta}-F(x,t))dx+||W||^2_{\dot{H}^{\frac{\alpha}{2}}}\Big)\\
        &\leq C\lambda^{\alpha}N^{-1+\alpha+2\epsilon}||W||^2_{L^2}+C||W||_{L^2}||F(x,t)||_{L^2} .
    \end{align*}
 and using the bounds for $||F(x,t)||_{L^2}$, using Duhamel's principle and choosing $\epsilon$ small we get, for the times considered, for any $\epsilon'>0$
\begin{equation*}
    ||W(t)||_{L^2}\leq e^{tC\lambda^{\alpha}N^{-1+\alpha+2\epsilon}} t||F||_{L^2}\leq C_{\epsilon'} N^{-\beta-1+\epsilon'}\lambda^{-1}.
\end{equation*}

On the other hand, by integrating in time \eqref{L2evol} we get, for any $t_{0}\in[0,(N\lambda)^{-\alpha}\ln(N)]\cap [0,T_{crit}]$, any $\epsilon>0$

\begin{equation}\label{controlha}
  \int_{0}^{t_{0}}||W||^2_{H^{\frac{\alpha}{2}}}\leq \text{sup}_{t\in[0,t_{0}]}(N\lambda)^{-\alpha}\ln(N)(||F||_{L^2}||W||_{L^2}+||W||^2_{L^2})\leq C \frac{N^{\epsilon}}{(N^{\beta+1}\lambda^{\beta})^2 }.  
\end{equation}

To bound the growth  of the higher order norm, we note that

\begin{align}\label{evolucionhs}
    &\frac{\partial}{\partial t}||\Lambda^{s}W||^2_{L^2}=-2\int_{\mathds{R}^2}\Lambda^{s}(W)\Lambda^{s}[v(W)\cdot\nabla (W+\bar{w}_{N,\beta}(x,t))+v(\bar{w}_{N,\beta}(x,t))\cdot\nabla W+\Lambda^{\alpha} W-F(x,t)]dx.\\\nonumber
\end{align}

In order to bound the growth of the $H^{s}$ norm with $s=2-\alpha+\delta$ (we will now just write $s$ instead of $2-\alpha+\delta$ for compactness of notation), we need to bound each of these terms under the assumption that \eqref{controlha} and \eqref{controll2} hold. First, we note that, as seen in \cite{Ju},

$$||\Lambda^{s}(W)\Lambda^{s}[v(W)\cdot\nabla (W)]||_{L^1}\leq C ||\Lambda^{2-\frac{\alpha}{2}}W||_{L^2}||\Lambda^{s+\frac{\alpha}{2}}W||_{L^2}||\Lambda^{s}W||_{L^2}$$
so, using our hypothesis for $||W||_{H^{s}}$, \eqref{controll2}, the interpolation inequality for Sobolev spaces and some basic computations we get, for $N$ big
\begin{align*}
    &||\Lambda^{s}(W)[v(W)\cdot\nabla (W)]||_{L^1}-\frac{1}{100}||\Lambda^{s+\frac{\alpha}{2}}W||^2_{L^2}\\
    &\leq C ||\Lambda^{2-\frac{\alpha}{2}}W||_{L^2}||\Lambda^{s+\frac{\alpha}{2}}W||_{L^2}||\Lambda^{s}W||_{L^2}-\frac{1}{100}||\Lambda^{s+\frac{\alpha}{2}}W||^2_{L^2}\\
    &\leq ||W||_{H^{s+\frac{\alpha}{2}}}||W||_{H^{s+\frac{\alpha}{2}-\delta}}-\frac{1}{100}||W||^2_{H^{s+\frac{\alpha}{2}}}\leq 0\\
\end{align*}
and using our hypothesis for $F$ and taking $\epsilon$ and $\delta$ small enough we get

$$||\Lambda^{s}(W)\Lambda^{s}(F)||_{L^1}\leq ||\Lambda^{s}(W)||_{L^2}\frac{N^{\epsilon} (\lambda N)^{2+\delta}}{N^{\beta+1}\lambda^{\beta}}\leq  C||\Lambda^{s}(W)||_{L^2}(\lambda N)^{\alpha} N^{-\delta}.$$

For the rest of the terms, we need to use a technical lemma from \cite{Dongli}

\begin{lemma}\label{katoponce}
Let $s > 0$. Then for any $s_{1}, s_{2} \geq 0$ with $s_{1} + s_{2} = s$, and any $f$, $g \in \mathcal{S}(\mathds{R}^2)$, the following holds:
\begin{equation}\label{Dsconmutador}
    ||\Lambda^{s}(fg)-\sum_{|\mathbf{k}|< s_{1}}\frac{1}{\mathbf{k}!}\partial^{\mathbf{k}}f \Lambda^{s,\mathbf{k}}g-\sum_{|\mathbf{j}|\leq s_{2}}\frac{1}{\mathbf{j}!}\partial^{\mathbf{j}}g \Lambda^{s,\mathbf{j}}f||_{L^{2}}\leq C ||\Lambda^{s_{1}}f||_{L^{2}} ||\Lambda^{s_{2}}g||_{BMO}
\end{equation}
where $\mathbf{j}$ and $\mathbf{k}$ are multi-indexes, $\partial^{\mathbf{j}}=\frac{\partial}{\partial x^{j_{1}}_{1}\partial x^{j_{2}}_{2}}$, $\partial_{\xi}^{\mathbf{j}}=\frac{\partial}{\partial \xi^{j_{1}}_{1}\partial \xi^{j_{2}}_{2}}$ and $ \Lambda^{s,\mathbf{j}}$ is defined using

$$\widehat{\Lambda^{s,\mathbf{j}}f}(\xi)=\widehat{\Lambda^{s,\mathbf{j}}}(\xi)\hat{f}(\xi)$$
$$\widehat{\Lambda^{s,\mathbf{j}}}(\xi)=i^{-|\mathbf{j}|}\partial^{\mathbf{j}}_{\xi}(|\xi|^s).$$
\end{lemma}

Note that, even though this lemma is only valid for functions in $\mathcal{S}$ (the Schwartz space), an approximation argument allows us to use this lemma for the functions we will be considering.

We will also use that, for any $\epsilon,s_{1},s_{2}\geq 0$, 

$$||\bar{w}_{N,\beta}||_{C^{s_{1}}}\leq||\bar{w}_{pert}||_{C^{s_{1}}}+||\bar{g}(r,t)||_{C^{s_{1}}}\leq  C_{\epsilon}\frac{(N\lambda)^{s_{1}+1-\beta}}{N}N^{\epsilon}+C\lambda^{s_{1}+1-\beta}$$
which is just a direct application of \eqref{sobolevpert} and $||g(\lambda r,t)||_{C^{s_{1}}}\leq \lambda^{s}||g(r,t)||_{C^{s_{1}}}$. Furthermore, this implies that, for $s_{1},s_{2}\geq 0$
$$||\Lambda^{s_{1}}\bar{w}_{N,\beta}||_{C^{s_{2}}}\leq C_{\epsilon}\frac{(N\lambda)^{s_{1}+s_{2}+1-\beta}}{N}N^{\epsilon}+C\lambda^{s_{1}+s_{2}+1-\beta}$$
and if $|\mathbf{k}|=1$, $s_{2}\geq 1$,
$$||\Lambda^{s_{2},\mathbf{k}}\bar{w}_{N,\beta}||_{L^{\infty}}\leq N^{\epsilon} \frac{(\lambda N)^{s_{2}-\beta}}{N}+ C\lambda^{s_{1}-\beta}$$
which follow directly from
$$||\Lambda^{s_{1}}f||_{C^{k}}\leq C_{k,\epsilon}||f||_{C^{k+s_{1}+\epsilon}},||\Lambda^{s_{2},\mathbf{k}}f||_{C^{k}}\leq C_{k,\epsilon}||f||_{C^{k+s_{2}-1+\epsilon}}$$
for $|\mathbf{k}|=1$, $k>0$, $s_{2}\geq 1,s_{1}\geq0$.
 Now, applying lemma \ref{katoponce} with $s_{1}=1+\delta$, $W=f$ we get
\begin{align*}
    &||\Lambda^{s}(W)\Lambda^{s}[v(W)\cdot\nabla \bar{w}_{N,\beta}(x,t)]||_{L^1}\\
    &\leq ||\Lambda^{s}(W)\Lambda^{s}[v(W)]\cdot\nabla \bar{w}_{N,\beta}(x,t)]||_{L^1}+||\Lambda^{s}(W)v(W)\cdot\Lambda^{s}[\nabla \bar{w}_{N,\beta}(x,t)]||_{L^1}\\
    &+||\Lambda^{s}W||_{L^2}(\sum_{|\mathbf{k}|= 1}||\partial^{\mathbf{k}}W||_{L^2}||\Lambda^{s,\mathbf{k}}\nabla\bar{w}_{n,\beta}||_{L^{\infty}}+||\Lambda^{1+\delta}W||_{L^2}||\Lambda^{1-\alpha}\nabla\bar{w}_{n,\beta}||_{L^{\infty}}).
\end{align*}

But, for $t\in[0,(\lambda N)^{-\alpha}(\ln(N))^2]$, using the bounds for $\bar{w}_{N,\beta}$ and the interpolation inequality and taking $\delta$ small and $N$ big

\begin{align*}
    &\int_{0}^{t}||\Lambda^{s}(W)\Lambda^{s}[v(W)]\cdot\nabla \bar{w}_{N,\beta}(x,\tau)]||^2_{L^2}d\tau\leq C\lambda^{2-\beta}\int_{0}^{t}||\Lambda^{\frac{\alpha}{2}}W||^{\frac{\alpha}{s}}_{L^2}||\Lambda^{s+\frac{\alpha}{2}}W||^{\frac{2s-\alpha}{s}}_{L^2}dt\\
    &\leq  C\lambda^{2-\beta}(\int_{0}^{t}||\Lambda^{\frac{\alpha}{2}}W||^2_{L^2}dt)^{\frac{\alpha}{2s}}(\int_{0}^{t}||\Lambda^{s+\frac{\alpha}{2}}W||^2_{L^2}dt)^{\frac{2s-\alpha}{2s}}\leq C_{\epsilon}N^{\epsilon}\lambda^{2-\beta}(N^{\beta+1}\lambda^{\beta})^{\frac{-\alpha}{s}}(\int_{0}^{t}||\Lambda^{s+\frac{\alpha}{2}}W||^2_{L^2}dt)^{\frac{2s-\alpha}{2s}}\\
    &\leq C_{\epsilon} N^{\frac{\alpha}{2-\alpha}(1-\beta)+o(\delta)+o(\epsilon)}(\int_{0}^{t}||\Lambda^{s+\frac{\alpha}{2}}W||^2_{L^2}dt)^{\frac{2s-\alpha}{2s}}\leq \frac {1}{100}(\int_{0}^{t}||\Lambda^{s+\frac{\alpha}{2}}W||^2_{L^2}dt)^{\frac{2s-\alpha}{2s}} .
\end{align*}

Similarly, for $\delta$ small and $N$ big 

\begin{align*}
    &\int_{0}^{t}||\Lambda^{s}(W)v(W)\cdot\Lambda^{s}[\nabla \bar{w}_{N,\beta}(x,t)]||_{L^2}d\tau\leq C_{\epsilon}\lambda^{s+2-\beta}N^{s+1-\beta+\epsilon}\int_{0}^{t}||\Lambda^{s}W||_{L^2}||W||_{L^2}d\tau\\
    &\leq C_{\epsilon}\lambda^{s+2-\beta}N^{s+1-\beta+\epsilon}(\lambda^{\beta}N^{\beta+1})^{-1}N^{\epsilon}(N\lambda)^{-\frac{\alpha}{2}}N^{\epsilon}(N^{\beta+1}\lambda^{\beta})^{-\frac{\alpha}{2s}}(\int_{0}^{t}||\Lambda^{s+\frac{\alpha}{2}}W||^2_{L^2}d\tau)^{\frac{2s-\alpha}{4s}}\\
    &\leq C_{\epsilon} N^{-2(\beta-1)+o(\epsilon)+o(\delta)}(\int_{0}^{t}||\Lambda^{s+\frac{\alpha}{2}}W||^2_{L^2}d\tau)^{\frac{2s-\alpha}{4s}}\leq  \frac{1}{100} (\int_{0}^{t}||\Lambda^{s+\frac{\alpha}{2}}W||^2_{L^2}d\tau)^{\frac{2s-\alpha}{4s}}
\end{align*}
and
\begin{align*}
    &\int_{0}^{t}||\Lambda^{s}W||_{L^2}\sum_{|\mathbf{k}|\leq 1}||\partial^{\mathbf{k}}W||_{L^2}||\Lambda^{s,\mathbf{k}}\nabla\bar{w}_{n,\beta}||_{L^{\infty}}d\tau\leq C_{\epsilon}\lambda^{s+1-\beta}N^{s-\beta+\epsilon}\int_{0}^{t}||\Lambda^{1}W||_{L^2}||\Lambda^{s}W||_{L^2}d\tau\\
    &\leq C_{\epsilon}\lambda^{s+1-\beta}N^{s-\beta+\epsilon}N^{\epsilon}(N^{\beta+1}\lambda^{\beta})^{-[\frac{1}{2}+\frac{\alpha}{2s}]}(\int_{0}^{t}||\Lambda^{s+\frac{\alpha}{2}}W||_{L^2}d\tau)^{\frac{1}{2}+\frac{2s-\alpha}{2s}}\\
    &\leq C_{\epsilon} N^{1-\beta+o(\epsilon)+o(\delta)}\leq \frac{1}{100}(\int_{0}^{t}||\Lambda^{s+\frac{\alpha}{2}}W||_{L^2}d\tau)^{\frac{1}{2}+\frac{2s-\alpha}{2s}},\\
\end{align*}
\begin{align*}
    &\int_{0}^{t}||\Lambda^{1-\alpha}\nabla\bar{w}_{N,\beta}||_{L^{\infty}}||\Lambda^{1+\delta}W||_{L^2}||\Lambda^{s}W||_{L^2}d\tau\leq C_{\epsilon} \lambda^{3-\alpha-\beta}(1+N^{2-\alpha-\beta+\epsilon})\int_{0}^{t}||\Lambda^{1+\delta}W||_{L^2}||\Lambda^{s}W||_{L^2}d\tau\\
    &\leq C_{\epsilon}\lambda^{3-\alpha-\beta}N^{2-\alpha-\beta+\epsilon}N^{\epsilon}(N^{\beta+1}\lambda^{\beta})^{\frac{1-\frac{\alpha}{2}-\delta}{2-\alpha}+\frac{\alpha}{2s}}(\int_{0}^{t}||\Lambda^{s+\frac{\alpha}{2}}W||_{L^2}d\tau)^{\frac{1-\frac{\alpha}{2}+\delta}{2-\alpha}+\frac{2s-\alpha}{2s}}\\
    &\leq C_{\epsilon} N^{1-\beta+o(\epsilon)+o(\delta)} \leq  \frac{1}{100}(\int_{0}^{t}||\Lambda^{s+\frac{\alpha}{2}}W||_{L^2}d\tau)^{\frac{1-\frac{\alpha}{2}+\delta}{2-\alpha}+\frac{2s-\alpha}{2s}}.
\end{align*}

Similarly,

\begin{align*}
    &\int_{0}^{t}||\Lambda^{s}(W)\Lambda^{s}[v(\bar{w}_{N,\beta}(x,t))\cdot\nabla W]||_{L^1}d\tau\\
    &\leq \frac{1}{100}((\int_{0}^{t}||\Lambda^{s+\frac{\alpha}{2}}W||_{L^2}d\tau)^{\frac{1}{2}+\frac{2s-\alpha}{2s}}+2(\int_{0}^{t}||\Lambda^{s+\frac{\alpha}{2}}W||^2_{L^2}d\tau)^{\frac{2s-\alpha}{2s}}).
\end{align*}

Therefore, integrating \eqref{evolucionhs} we get that, for $t\in[0,T_{crit}]\cap[0,(N\lambda)^{-\alpha}(\ln(N))^2]$

$$||\Lambda^{s}W||^2_{L^2}\leq \frac{1}{10},||W||^2_{H^s}\leq \frac{1}{5}$$
and, since by continuity of the $H^s$ norm we must have that $||W(x,T_{crit})||^2_{H^s}=1$,  in particular it must be that $T_{crit}>(N\lambda)^{-\alpha}(\ln(N))^2$ as we wanted to prove.

\end{proof}

Lemma \ref{shortcontrol} allows us to show that our pseudo-solution is a very good approximation of the actual solution for $t\in[0,(\lambda N)^{-\alpha}(\ln(N))^2]$. Furthermore, at $t^{*}=(\lambda N)^{-\alpha}(\ln(N))^2$, we have that

$$||\bar{w}_{pert}(x,t^{*})||_{L^2}\leq  \frac{C}{(N\lambda)^{\beta}} \text{sup}_{r\in \text{supp}[\bar{w}_{pert}]}e^{-G(r,t^{*})}\leq \frac{C}{(N \lambda )^{\beta}}e^{C(\lambda N)^{\alpha} (\lambda N)^{-\alpha}(\ln(N))^2}\leq \frac{C}{N^{\beta+2}\lambda^{\beta}} $$
and for any $\epsilon>0$
\begin{align*}
    &||\bar{w}_{pert}(x,t^{*})||_{H^2}\leq  C_{\epsilon}(N\lambda)^{2-\beta} N^{\epsilon}\text{sup}_{r\in \text{supp}[\bar{w}_{pert}]}e^{-G(r,t^{*})}\leq C(N \lambda )^{2-\beta}N^{\epsilon}e^{C(\lambda N)^{\alpha} (\lambda N)^{-\alpha}(\ln(N))^2}\\
    &\leq C_{\epsilon}N^{-\beta+\epsilon}\lambda^{2-\beta}
\end{align*}
so in particular, by interpolation, for small $\delta>0$
\begin{equation}
    ||w_{pert}(x,t)||_{H^{2-\alpha+\delta}}\leq N^{-\delta}.
\end{equation}

Combining this with \eqref{controlhs} and \eqref{controll2}, we have that, for $t=(\lambda N)^{-\alpha}(\ln(N))^2$

$$||w_{N,\beta}(r,\theta,t)-\bar{g}(r,t)||_{L^2}\leq \frac{C}{N^{\beta+1}\lambda^{\beta}},$$

$$||w_{N,\beta}(r,\theta,t)-\bar{g}(r,t)||_{H^{2-\alpha+\delta}}\leq N^{-\delta}.$$

With this information, we can obtain the following lemma.

\begin{lemma}\label{longcontrol}
There is $N_{0},\epsilon_{0},\delta $ such that if $N\geq N_{0}$, $\epsilon\leq \epsilon_{0}$ then  for any $t\in[(N\lambda)^{-\alpha}(\ln(N))^2,\lambda^{-\alpha}(\ln(N))^3]$ we have that

\begin{equation}\label{longcontroll2}
    ||w_{N,\beta}(x,t)-\bar{g}(r,t)||_{L^2}\leq C_{\epsilon}\frac{N^{\epsilon}}{N^{\beta+1}\lambda^{\beta}},
\end{equation}

\begin{equation}\label{longcontrolhs}
    ||w_{N,\beta}(x,t)-\bar{g}(r,t)||_{H^{2-\alpha+\delta}}\leq 2.
\end{equation}
where $w_{N,\beta}$ is the solution \eqref{SQG} with the same initial conditions as $\bar{w}_{N,\beta}$ and $\bar{g}(r,t)$ is as in \eqref{radial+pert}.
\end{lemma}

\begin{proof}
    We will omit the proof since it is completely analogous to that of lemma \ref{shortcontrol}: We note that \eqref{longcontroll2} and \eqref{longcontrolhs} hold for $t=(N\lambda)^{-\alpha}(\ln(N))^2$, we prove \eqref{longcontroll2} for times when \eqref{longcontrolhs} holds, then we obtain an inequality like \eqref{controlha} and with that, the evolution equation for $w_{N,\beta}(x,t)-\bar{g}(r,t)$ and lemma \ref{katoponce} we prove that \eqref{longcontrolhs} holds for $t\in[(N\lambda)^{-\alpha}(\ln(N))^2,\lambda^{-\alpha}(\ln(N))^3]$.
\end{proof}

\begin{corollary}\label{smallnorm}
There is $N_{0},\delta>0$ such that if $N\geq N_{0}$,  then  for any $t\in[\lambda^{-\alpha}(\ln(N))^3,N^{\delta}]$ we have that

$$||w_{N,\beta}(x,t)||_{H^{2-\alpha+\delta}}\leq CN^{-\delta}.$$
and for $t\in[N^{-\frac{\delta}{2}},N^{\delta}]$
\begin{equation}\label{cotac1}
    ||w_{N,\beta}(x,t)||_{H^{2+\delta}}\leq CN^{-\frac{\delta}{2}},
\end{equation}
where $w_{N,\beta}$ is the solution to \eqref{SQG} with the same initial conditions as $\bar{w}_{N,\beta}$.

\end{corollary}

\begin{proof}
First, lemma \ref{longcontrol} allows us to show that, for $t=\lambda^{-\alpha}(\ln(N))^3$ and for  small $\delta>0$, any $\epsilon>0$

$$||w_{N,\beta}(x,t)-\bar{g}(r,t)||_{H^{2-\alpha+\delta}}\leq 2$$
$$||w_{N,\beta}(x,t)-\bar{g}(r,t)||_{L^2}\leq \frac{C_{\epsilon}N^{\epsilon}}{\lambda^{\beta}N^{\beta+1}}$$
and, using that $\text{supp } \hat{\bar{g}}(\hat{r},0)\subset \{\hat{r}:\hat{r}\in(\lambda c,\infty)\}$ for some $c>0$ (see lemma \ref{initialvel}), again for the same time

$$||\bar{g}(r,t)||_{H^s}\leq e^{C\lambda^{-\alpha}t}||\bar{g}(r,0)||_{H^{s}}=C\frac{\lambda^{s-\beta}}{N^3}\leq \frac{C\lambda^{s-2+\alpha}}{N^{2}} $$
and combining the three inequalities, taking $\epsilon$ small and using the interpolation inequality, we get that there is a small $\bar{\delta}>0$ such that, for  $t=\lambda^{-\alpha}(\ln(N))^3$
$$||w_{N,\beta}(x,t)||_{H^{2+\alpha-\bar{\delta}}}\leq CN^{-\bar{\delta}}.$$

Finally, we note that

\begin{align*}
    &\frac{\partial}{\partial t}||w_{N,\beta}||^2_{\dot{H}^{2-\alpha+\bar{\delta}}}\leq 2 ||\Lambda^{2-\alpha+\bar{\delta}}(w_{N,\beta})\Lambda^{2-\alpha+\bar{\delta}}[v(w_{N,\beta})\cdot\nabla w_{N,\beta}]||_{L^1}-2||\Lambda^{2-\frac{\alpha}{2}+\bar{\delta}}w_{N,\beta}||^2_{L^2}\\
    &\leq C ||\Lambda^{2-\frac{\alpha}{2}}w_{N,\beta}||_{L^2}||\Lambda^{2-\frac{\alpha}{2}+\bar{\delta}}w_{N,\beta}||_{L^2}||\Lambda^{2-\alpha+\bar{\delta}}w_{N,\beta}||_{L^2}-2||\Lambda^{2-\frac{\alpha}{2}+\bar{\delta}}w_{N,\beta}||^2_{L^2}\\
    &\leq  C (||\Lambda^{2-\frac{\alpha}{2}+\bar{\delta}}w_{N,\beta}||_{L^2}+||w_{N,\beta}||_{L^2})||\Lambda^{2-\frac{\alpha}{2}+\bar{\delta}}w_{N,\beta}||_{L^2}||\Lambda^{2-\alpha+\bar{\delta}}w_{N,\beta}||_{L^2}-2||\Lambda^{2-\frac{\alpha}{2}+\bar{\delta}}w_{N,\beta}||^2_{L^2}\\
    &\leq (C||\Lambda^{2-\alpha+\bar{\delta}}w_{N,\beta}||_{L^2}-1)||\Lambda^{2-\frac{\alpha}{2}+\bar{\delta}}w_{N,\beta}||^2_{L^2}+C||\Lambda^{2-\alpha+\bar{\delta}}w_{N,\beta}||^2_{L^2}||w_{N,\beta}||^2_{L^2}
\end{align*}
and integrating in time gives the desired bound.

Next, using the interpolation inequality and our bounds for $w_{N,\beta}$, we have that, for $t\in[\lambda^{-\alpha}(\ln(N))^3,N^{\delta}]$

\begin{align*}
    &\frac{\partial }{\partial t}||w_{N,\beta}(x,t)||^2_{\dot{H}^{2+\delta}}\leq C||w_{N,\beta}(x,t)||^3_{H^{2+\delta}}-2||w_{N,\beta}(x,t)||^2_{\dot{H}^{2+\delta+\frac{\alpha}{2}}}\\
    &\leq C||w_{N,\beta}(x,t)||^3_{H^{2+\delta}}+\frac{C}{\lambda^{2\beta}N^{4}}-2||w_{N,\beta}(x,t)||^2_{H^{2+\delta+\frac{\alpha}{2}}}\\
    &\leq C||w_{N,\beta}(x,t)||^3_{H^{2+\delta}}+\frac{C}{\lambda^{2\beta}N^{4}}-2\frac{||w_{N,\beta}(x,t)||^3_{H^{2+\delta}}}{||w_{N,\beta}(x,t)||_{H^{2-\alpha+\delta}}}\leq -CN^{\delta}||w_{N,\beta}(x,t)||^3_{H^{2+\delta}}+\frac{C}{\lambda^{2\beta}N^{4}}.
\end{align*}
Now, we note that if for some $t_{0}\in[\lambda^{-\alpha}(\ln(N))^3,N^{\delta}]$ we have $||w_{N,\beta}(x,t)||_{H^{2+\delta}}\leq N^{-1}$, then \eqref{cotac1} holds trivially for $t\in[t_{0},N^{\delta}]$ by integrating the equation. Therefore, it is enough to study the behaviour for $t$ such that $||w_{N,\beta}(x,t)||_{H^{2+\delta}}\geq N^{-1}$, which in particular gives, for $N$ big
$$N^{\delta}||w_{N,\beta}(x,t)||^3_{H^{2+\delta}}>>\frac{1}{\lambda^{2\beta}N^{4}}.$$
and with this, we have

$$\frac{\partial}{\partial t}||w_{N,\beta}(x,t)||_{H^{2+\delta}}\leq -CN^{\delta}||w_{N,\beta}(x,t)||^2_{H^{2+\delta}}$$
and integrating this equation gives the desired result.

\end{proof}

The last computation we need to do regarding $w_{N,\beta}$ is to show that it exhibits a very fast growth in the $H^{\beta}$ norm.

\begin{lemma}\label{growth}
  There is $N_{0}$ such that if $N\geq N_{0}$ then
  $$||w_{N,\beta}(x,t=\lambda^{-2+\beta}(\ln(N))^{\frac12})||_{\dot{H}^{\beta}}\geq C(\ln(N))^{\frac{\beta}{2}}||w_{N,\beta}(x,0)||_{H^{\beta}}$$
  with $w_{N,\beta}$ the solution to \eqref{SQG} with the same initial conditions as $\bar{w}_{N,\beta}$.
\end{lemma}

\begin{proof}
    We start by noting that, by using the scaling properties of the norms $\dot{H}^{s}$ plus the definition of $\bar{g}(r,t)$
    $$||\bar{g}(r,t)||_{H^{s}}\leq C_{s}\lambda^{s-\beta}.$$
    On the other hand, we have that, for $i=1,2$, by direct computation
    $$||w_{N,\beta}(x,0)-\bar{g}(r,0)||_{H^i}\leq C\frac{(\lambda N)^{i}}{\lambda^{\beta} N^{\beta}},$$
    and combining these inequalities plus the interpolation inequality gives us
    $$||w_{N,\beta}(x,0)||_{H^{\beta}}\leq C .$$
    On the other hand for $t=\lambda^{-2+\beta}(\ln(N))^{\frac12})$, using lemma \ref{shortcontrol} we have
    $$||w_{N,\beta}(x,t)-\bar{g}(r,t)||_{L^2}\leq \frac{C}{\lambda^{\beta} N^{\beta}}$$
    and
    \begin{align*}
        &||w_{N,\beta}(x,t)-\bar{g}(r,t)||_{\dot{H}^1}\geq ||\bar{w}_{pert}||_{\dot{H}^1}-C_{\epsilon}\frac{N^{\epsilon}}{N^{\beta+1}\lambda^{\beta}}.
    \end{align*}
    Furthermore using the expression for $\frac{\partial }{\partial x_{1}}$ in polar coordinates for $t\in[0,\lambda^{-2+\beta}(\ln(N))^{\frac12}]$
    $$ ||1_{\text{supp}(f(\lambda r))}\partial_{r}\Theta(r,t)||_{L^{\infty}}\leq C (\ln(N))^{\frac12}$$
    so for $r\in \text{supp}(f(\lambda r))$ and $t=\lambda^{-2+\beta}(\ln(N))^{\frac12}$ $G(r,t)\leq C$, which, after some basic computations, gives, for $t=\lambda^{-2+\beta}(\ln(N))^{\frac12}$
    $$|f(\lambda r)\partial_{r}\Theta(r,t)e^{-G(r,t)}|\geq C |f(\lambda r)|(\ln(N))^{\frac12}.$$
    Using all this we get, for $N$ big,

    \begin{align*}
        &||\bar{w}_{pert}||_{\dot{H}^1}\geq ||\frac{\partial}{\partial x_{1}}\bar{w}_{pert}||_{L^2}\geq ||cos(\theta)\frac{\partial}{\partial r}\bar{w}_{pert}||_{L^2}-\frac{C}{(N\lambda)^{\beta-1}}\\
        &\geq ||cos(\theta)f(\lambda r) \lambda N \big(\frac{\partial}{\partial r}\Theta(r,t)\big) \frac{\sin(N(\theta+\Theta(r,t))-C_{0}\int_{0}^{t}\frac{\partial \bar{g}(r,s)}{\partial r}ds) }{N^{\beta}\lambda^{\beta}}e^{-G(r,t)}||_{L^2}-\frac{C}{(N\lambda)^{\beta-1}}\\
        &\geq \frac{C}{(N\lambda)^{\beta-1}}||f(\lambda r)\big(\frac{\partial}{\partial r}\Theta(r,t)\big) e^{-G(r,t)}||_{L^2}-\frac{C}{(N\lambda)^{\beta-1}}\geq \frac{C (\ln(N))^{\frac{1}{2}}}{(N\lambda)^{\beta-1}},
    \end{align*}
   and using the interpolation inequality we get
   $$||\bar{w}_{pert}||_{\dot{H}^{\beta}}\geq \frac{||\bar{w}_{pert}||^{\beta}_{\dot{H}^1}}{||\bar{w}_{pert}||^{\beta-1}_{L^2}}\geq C (\ln(N))^{\frac{\beta}{2}},$$
   so
   $$\frac{||w_{N,\beta}(x,t)||_{\dot{H}^{\beta}}}{||w_{N,\beta}(x,0)||_{H^{\beta}}}\geq C(\ln(N))^{\frac{\beta}{2}}$$
   as we wanted to prove.
\end{proof}
\begin{remark}
    Lemma \ref{growth} would give us the growth around zero of the $H^{\beta}$ norm, namely the final solution obtained in theorem \ref{final} will have a sequence of times $t_{n}$ such that

    $$\text{lim}_{t_{n}\rightarrow 0}\frac{||w(x,t_{n})||_{H^{\beta}}}{|\ln(t_{n})|^{\frac{\beta}{2}}}>0.$$
    It should be noted that it is not the goal of this paper to try and obtain the optimal explosion rate around  $t=0$, and this rate can probably be improved substantially.
\end{remark}

\section{Loss of regularity}

We will now use the previous results to obtain a more compact and usable theorem before we go on to prove the main theorem.

\begin{theorem}\label{wn}
For any $n\in\mathds{N}$, $\alpha\in(0,1)$, $\beta\in(1,2-\alpha)$, there exists initial conditions $w_{n}(x,0)$, a solution $w_{n}(x,t)$ to \eqref{SQG} and $t_{n}\in[0,2^{-n}]$  such that
$$||w_{n}(x,0)||_{H^{\beta}}\leq 2^{-n},||w_{n}(x,t_{n})||_{H^{\beta}}\geq 2^{n}.$$
Furthermore, there is small $\delta>0$ such that, for $t\in [\frac{1}{n},1]$
$$||w_{n}(x,t)||_{H^{2+\delta}}\leq 2^{-n}$$
and for $t\in[0,1]$
$$||w_{n}(x,t)||_{H^{1+\delta}}\leq 2^{-n},||w_{n}(x,t)||_{L^1}\leq 2^{-n}$$
\begin{equation}\label{c1n}
    ||w_{n}(x,t)||_{H^{6}}\leq C_{n}.
\end{equation}

\end{theorem}

\begin{proof}
    We start by fixing $\alpha\in(0,1)$ and $\beta\in(1,2-\alpha)$, and we consider the solutions $w_{N,\beta}(x,t)$ that we considered earlier. These solutions fulfil that 
    \begin{itemize}
        \item $||w_{N,\beta}(x,0)||_{H^{\beta}}\leq C.$
        \item For $t=\lambda^{-2+\beta}\ln(N)$, $N$ big 
        $$||w_{N,\beta}(x,t)||_{H^{\beta}}\geq C (\ln(N))^{\frac{\beta}{2}}.$$
        \item There is some small $\delta>0$ such that, for $N$ big and  $t\in [\lambda^{-\alpha}(\ln(N))^3,N^{\delta}]$
        $$||w_{N,\beta}(x,t)||_{H^{2+\delta}}\leq C N^{-\frac{\delta}{2}}.$$
        \item There is a small $\delta>0$ such that for $t\in[0,N^{\delta}]$,
        $$||w_{N,\beta}||_{H^{1+\delta}}\leq C N^{-\delta},||w_{N,\beta}||_{L^1}\leq C N^{-\delta}.$$
    \end{itemize}
    If we now consider 
    $$w_{N,K}(x,t):=\frac{w_{N,\beta}(K x, K^{\alpha} t)}{K^{1-\alpha}}$$
    we have that these functions are also solutions to \eqref{SQG} and,
    \begin{itemize}
        \item $||w_{N,K}(x,0)||_{H^{\beta}}\leq CK^{-2+\alpha+\beta}.$
        \item For $t=K^{-\alpha}\lambda^{-2+\beta}\ln(N)$, $N$ big 
        $$||w_{N,K}(x,t)||_{H^{\beta}}\geq C K^{2-\alpha-\beta}(\ln(N))^{\frac{\beta}{2}}.$$
        \item There is some small $\delta>0$ such that, for $N$ big and  $t\in [K^{-\alpha}\lambda^{-\alpha}(\ln(N))^3,K^{-\alpha}N^{\delta}]$
        $$||w_{N,K}(x,t)||_{H^{2+\delta}}\leq C K^{\frac{\delta}{2}}N^{-\frac{\delta}{2}}.$$
        \item There is a small $\delta>0$ such that, for $t\in[0,K^{-\alpha}N^{\delta}]$, 
        $$||w_{N,K}||_{H^{1+\delta}}\leq C K^{1+\delta-2+\alpha}N^{-\delta},||w_{N,K}||_{L^1}\leq C K^{-3+\alpha}N^{-\delta}.$$
    \end{itemize}
    so, for any fixed $n$, taking $K$ big and then $N$ big gives us all the inequalities but the bound in $H^{6}$. But then using Theorem 3.1 from \cite{ConstantinWu} tells us that $w_{N,K}$ is in $C^{\infty}$ for any $t>0$, and since our initial conditions are in $C^{\infty}$, using the continuity of the $H^{6}$ norm for \eqref{SQG} finishes the proof.
\end{proof}






We are now ready to prove the main theorem of this paper.

\begin{theorem}\label{final}
    Given $\epsilon>0$, $\alpha\in(0,1)$, $\beta\in(1,2-\alpha)$, there exists initial conditions $w(x,0)$ with $||w(x,0)||_{H^{\beta}}\leq \epsilon$ and a solution $w(x,t)$ to \eqref{SQG} such that $w(x,t)\in C^{\infty}$ for $t\in (0,\infty)$ and there exists a sequence of times $(t_{n})_{n\in\mathds{N}}$ converging to zero, with
    $\text{lim}_{n\rightarrow\infty}||w(x,t_{n})||_{H^{\beta}}=\infty.$
    Furthermore, this is the only solution with the given initial conditions that is in $L_{t}^{\infty}H^{1}_{x}$.
\end{theorem}

\begin{proof}
    For this proof, we will be considering initial conditions of the form
    \begin{equation}\label{initialtotal}
        w(x,0)=\sum_{j=1}^{\infty}T_{R_{j}}(w_{c_{j}}(x,0))
    \end{equation}
    with $T_{R}(f(x_{1},x_{2}))=f(x_{1}+R,x_{2})$, $w_{c_{J}}(x,0)$ the initial conditions given by theorem \ref{wn}.
    In order to show properties of the solution given by \eqref{initialtotal}, we will also consider a truncated initial conditions 
    \begin{equation}\label{initialtrunc}
        \tilde{w}_{J}(x,0)=\sum_{j=1}^{J}T_{R_{j}}(w_{c_{j}}(x,0))
    \end{equation}
    and we will refer the solution with initial conditions given by \eqref{initialtrunc} as $\tilde{w}_{J}(x,t)$.
    Fixed $\epsilon,$ which we will assume $\frac{1}{2}>\epsilon$ without loss of generality, we will choose  $(c_{j})_{j\in\mathds{N}}$ so that they fulfil:
    \begin{itemize}
        \item $c_{i}>c_{j}$ if $i>j$. 
        \item $2^{-c_{1}+1}\leq \epsilon$, so $||w(x,0)||_{H^{\beta}}\leq \epsilon$.
        \item If we define
        \begin{equation}\label{sj}
            S_{j}:=\sum_{i=1}^{j}C_{c_{i}}
        \end{equation}
        with $C_{c_{i}}$ the constants given by \eqref{c1n}, then we take $c_{j}$ so that
        $$c_{j}\geq j e^{S_{j-1}+1},2^{c_{j}}\geq 2 S_{j-1}.$$
        \item If $t_{c_{j}}$ is the time given by theorem \ref{wn} such that
        $$||w_{c_{j}}(x,t_{c_{j}})||_{H^{\beta}}\geq 2^{c_{j}}$$
        then $\frac{1}{c_{j+1}}\leq t_{c_{j}}$.
    \end{itemize}

    We will now divide the proof in four different steps.
    
    Step 1)The goal of this step is to show the following claim:
    
    For any choice of $(c_{j})_{j\in\mathds{N}}$ and $\epsilon'>0$, we can choose $(R_{j})_{j\in\mathds{N}}$ such that, for $t\in[0,1]$,  for any $J\in\mathds{N}$
    \begin{equation}\label{convh5}
        ||\tilde{w}_{J}(x,t)-\tilde{w}_{J-1}(x,t)-T_{R_{J}}(w_{c_{J}}(x,t))||_{H^{5}}\leq \epsilon' 2^{-J-1},
    \end{equation}
    and $\tilde{w}_{J}(x,t)\in H^{6}$ for $t\in[0,1]$, and such that for $t\in[0,1]$
    \begin{equation}\label{convc3}
        ||\tilde{w}_{J}(x,t)-\tilde{w}_{J-1}(x,t)||_{C_{x}^{3}(B_{J}(0))}\leq \epsilon'2^{-J}
    \end{equation}
    where $B_{J}(0)$ is the ball of radius $J$ centered at the origin.

    We will show \eqref{convh5} and \eqref{convc3} by induction, by showing that, for any $J\in\mathds{N}$, fixed $(c_{j})_{j=1,...,J}$ and $\epsilon'>0$, we can choose $(R_{i})_{i=1,...,J}$ so that for any $j=1,...,J$
    \begin{equation}\label{convh5trunc}
        ||\tilde{w}_{j}(x,t)-\tilde{w}_{j-1}(x,t)-T_{R_{j}}(w_{c_{j}}(x,t))||_{H^{5}}\leq \epsilon' 2^{-j-1},
    \end{equation}
     \begin{equation}\label{convc3trunc}
        ||\tilde{w}_{j}(x,t)-\tilde{w}_{j-1}(x,t)||_{C_{x}^{3}(B_{J}(0))}\leq \epsilon'2^{-j}.
    \end{equation}
    and $\tilde{w}_{j}(x,t)\in H^{6}$.
     For $J=1$ this is trivial since $\tilde{w}_{J=0}=0$, $\tilde{w}_{J=1}(x,t)=T_{R_{1}}(w_{c_{1}}(x,t))$ and $w_{J=1}\in C^{\infty}$ for all time. Now, for some arbitrary $J$ we have that, if we define
    $$\bar{w}_{J}(x,t)=\tilde{w}_{J-1}(x,t)+T_{R_{J}}(w_{c_{J}}(x,t))$$ 
    we have that $\bar{w}_{J}(x,t)$ is a pseudo-solution for \eqref{SQG} with
    $$F(x,t)=-v(\tilde{w}_{J-1}(x,t))\cdot \nabla (T_{R_{J}}(w_{c_{J}}(x,t)))-v(T_{R_{J}}(w_{c_{J}}(x,t)))\cdot \nabla (\tilde{w}_{J-1}(x,t)).$$
    Furthermore, both $\tilde{w}_{J-1}(x,t)$ and $T_{R_{J}}(w_{c_{J}}(x,t))$ are $C^{\infty}$ functions for $t\in[0,1]$ since they are both solutions to \eqref{SQG} that are uniformly bounded in $H^{6}$, $\tilde{w}_{J-1}(x,t)$ by hypothesis and $T_{R_{J}}(w_{c_{J}}(x,t))$ by theorem \ref{wn}. But then we know that
    $\text{lim}_{R_{j}\rightarrow\infty}||F(x,t)||_{H^5}=0$ by using that for any two functions $f_{1}(x),f_{2}(x)\in H^{7}$
    $$\text{lim}_{R\rightarrow\infty}||f_{1}(x_{1},x_{2})f_{2}(x_{1}+R,x_{2})||_{H^5}=0,$$
    plus the fact that $C^{\infty}$ solutions solutions to \eqref{SQG} are continuous in time with respect to the $H^{7}$ norm.

    Now, to get \eqref{convh5trunc}, we just use that, if $\bar{w}_{error,c}(x,t)$ is a family of pseudo-solutions (which depends on the parameter $c$) with source term $F_{error,c}$ and fulfilling $||\bar{w}_{error,c}(x,t)||_{H^{6}}\leq C$ for $t\in[0,T]$ ($C$ independent of $c$), $||F_{error,c}(x,t)||_{H^{5}}\leq c$ and we call $w_{c}(x,t)$ the solution of \eqref{SQG} with the same initial conditions as $\bar{w}_{error,c}$, then 
    $$\lim_{c\rightarrow 0}||w_{c}(x,t)-\bar{w}_{error,c}(x,t)||_{H^{5}}=0,$$
and therefore,
\begin{equation}\label{rjh5}
    \text{lim}_{R_{J}\rightarrow\infty}|| \tilde{w}_{J}(x,t)-\tilde{w}_{J-1}(x,t)-T_{R_{J}}(w_{c_{J}}(x,t))||_{H^{5}}=0,
\end{equation}
and so taking $R_{J}$ big enough gives \eqref{convh5trunc}, and then since for $t\in[0,1]$ $w_{J}(x,t)$ is a $H^{5}$ solution to \eqref{SQG} with initial conditions in $C^{\infty}$, it must also be in $H^{6}$.

Next, for \eqref{convc3trunc}, since we only need to prove the case $j=J$, we  use
\begin{align*}
    &||\tilde{w}_{J}(x,t)-\tilde{w}_{J-1}(x,t)||_{C_{x}^{3}(B_{J}(0))}\\
    &\leq ||\tilde{w}_{J}(x,t)-\tilde{w}_{J-1}(x,t)-T_{R_{J}}(w_{c_{J}}(x,t)||_{H^{5}}+||T_{R_{J}}(w_{c_{J}}(x,t)||_{C_{x}^{3}(B_{J}(0))}\\
    &\leq \epsilon' 2^{-J-1}+||T_{R_{J}}(w_{c_{J}}(x,t)||_{C_{x}^{3}(B_{J}(0))}
\end{align*}
and, as before, using the continuity in time with respect to the $H^{5}$ norm of smooth solutions to \eqref{SQG} gives us
$$\text{sup}_{t\in[0,1]}\text{lim}_{R_{J}\rightarrow\infty}||T_{R_{J}}(w_{c_{J}}(x,t)||_{C_{x}^{3}(B_{J}(0))}=0,$$
so taking $R_{j}$ big enough finishes step 1.
    
    Step 2)The goal of this step is to obtain the properties of $\text{lim}_{J\rightarrow\infty}\tilde{w}_{J}(x,t)$:
    
    First we note that,
    $$||w_{c_{j}}(x,t)||_{H^{1+\delta}}\leq 2^{-c_{j}},||w_{c_{j}}(x,t)||_{L^1}\leq 2^{-c_{j}}$$
    so there exists
    $$w_{\infty}(x,t):=\lim_{J\rightarrow\infty}\tilde{w}_{J}(x,t)$$
    and  $\tilde{w}_{J}(x,t)$ tends to $w_{\infty}(x,t)$ in $ H^{1+\delta}\cap L^1$. 
    We would like to show that, for any $t\in[0,1]$
    $$\frac{\partial}{\partial t}w_{\infty}(x,t)+v(w_{\infty}(x,t))\cdot \nabla w_{\infty}(x,t)+\Lambda^{\alpha}(w_{\infty}(x,t))=0.$$

    For this, we note that, for $j>J$, using the properties of $\Lambda^{\alpha}$ and $v$, for $t\in[0,1]$
    \begin{align*}
        &||\Lambda^{\alpha}w_{\infty}(x,t)-\Lambda^{\alpha}\tilde{w}_{j}(x,t)||_{C_{x}^{1}(B_{J}(0))}\leq C( ||w_{\infty}(x,t)-\tilde{w}_{j}(x,t)||_{C_{x}^{2}(B_{J}(0))}+||w_{\infty}(x,t)-\tilde{w}_{j}(x,t)||_{L^1})\\
        &\leq C(\sum_{i=j}^{\infty}||\tilde{w}_{i+1}(x,t)-\tilde{w}_{i}(x,t)||_{C_{x}^{2}(B_{J}(0))}+2^{-j+1})\leq C 2^{-j},
    \end{align*}
    and similarly
    \begin{align*}
        &||v(w_{\infty})\cdot\nabla w_{\infty}-v(w_{j})\cdot\nabla w_{j}||_{C_{x}^{1}(B_{J}(0))}\\
        &\leq C\text{sup}_{i\geq j}||\tilde{w}_{i}||_{C_{x}^{2}(B_{J}(0))}(||w_{\infty}(x,t)-\tilde{w}_{j}(x,t)||_{C_{x}^{2}(B_{J}(0))}+||w_{\infty}(x,t)-w_{j}(x,t)||_{L^1})\\
        &+C\text{sup}_{i\geq j}(||\tilde{w}_{i}||_{C_{x}^{2}(B_{J}(0))}+||\tilde{w}_{i}||_{L^1})(||w_{\infty}(x,t)-\tilde{w}_{j}(x,t)||_{C_{x}^{2}(B_{J}(0))}\\
        &\leq C 2^{-j},
    \end{align*}
    so that

    \begin{align*}
        &w_{\infty}(x,t_{2})-w_{\infty}(x,t_{1})=\text{lim}_{J\rightarrow \infty}(\tilde{w}_{J}(x,t_{2})-\tilde{w}_{J}(x,t_{1}))\\
        &=-\text{lim}_{J\rightarrow \infty}\int_{t_{1}}^{t_{2}}(v(\tilde{w}_{J}(x,s))\cdot \nabla \tilde{w}_{J}(x,s)+\Lambda^{\alpha}(\tilde{w}_{J}(x,s)))ds\\
        &=-\int_{t_{1}}^{t_{2}}(v(w_{\infty}(x,s))\cdot \nabla w_{\infty}(x,s)+\Lambda^{\alpha}(w_{\infty}(x,s)))ds.
    \end{align*}    
    and using that
\begin{align*}
    &\text{lim}_{J\rightarrow\infty}||(v(w_{\infty}(x,t))\cdot \nabla w_{\infty}(x,t)+\Lambda^{\alpha}(w_{\infty}(x,t))-(v(\tilde{w}_{J}(x,t))\cdot \nabla \tilde{w}_{J}(x,t)\\
    &+\Lambda^{\alpha}(\tilde{w}_{J}(x,t))||_{C^{0}_{t}C^{1}_{x}([0,1]\times  B_{J}(0))}=0
\end{align*}
    and that $(v(\tilde{w}_{J}(x,t))\cdot \nabla \tilde{w}_{J}(x,t)+\Lambda^{\alpha}(\tilde{w}_{J}(x,t))$ is continuous in time with respect to the $C_{x}^{1}(B_{J}(0))$ norm, we get that the function
    $$v(w_{\infty}(x,t))\cdot \nabla w_{\infty}(x,t)+\Lambda^{\alpha}(w_{\infty}(x,t))$$
    is also continuous in time (for $t\in[0,1]$) with respect to the $C^{1}_{x}(B_{J}(0))$ norm, so
    \begin{align*}
        &\frac{\partial }{\partial t}w_{\infty}(x,t)=-(v(w_{\infty}(x,t))\cdot \nabla w_{\infty}(x,t)+\Lambda^{\alpha}(w_{\infty}(x,t))).
    \end{align*} 
    holds.

    Furthermore, for $t>0$
    \begin{align*}
        &||w_{\infty}(x,t)-\tilde{w}_{J}(x,t)||_{H^{2+\delta}}\leq \sum_{j=J}^{\infty}||w_{\infty}(x,t)-\tilde{w}_{J}(x,t)-T_{R_{j}}(w_{c_{j}}(x,t))||_{H^{5}}+\sum_{j=J}^{\infty}||w_{c_{j}}(x,t)||_{H^{2+\delta}}\\
        &\leq \epsilon'2^{-J-1}+\sum_{j=J}^{\infty}||w_{c_{j}}(x,t)||_{H^{2+\delta}}.
    \end{align*}
    
    Now, we note that, if $t\geq \frac{1}{c_{j}}$, then
    $$||w_{c_{j}}(x,t)||_{H^{2+\delta}}\leq 2^{-j}$$
    so, if  
    \begin{equation}\label{j0}
        j_{0}(t):=\text{max}(\{j\in\mathds{N}: \frac{1}{c_{j}}< t\},\{0\})
    \end{equation}
    then
    $$\sum_{j=1}^{\infty}||w_{c_{j}}(x,t)||_{H^{2+\delta}}\leq \sum_{j=j_{0}(t)+1}^{\infty}2^{-c_{j}}+S_{j_{0}(t)}\leq S_{j_{0}(t)}+1$$
    with $S_{j}$ as in \eqref{sj} and if $J>j_{0}(t)$
     $$\sum_{j=J}^{\infty}||w_{c_{j}}(x,t)||_{H^{2+\delta}}\leq \sum_{j=J}^{\infty}2^{-c_{j}}\leq 2^{-J+1}.$$
    This in particular means that, for any $t>0$,
    \begin{equation}\label{convh2del}
        \text{lim}_{J\rightarrow\infty}||w_{\infty}(x,t)-\tilde{w}_{J}(x,t)||_{H^{2+\delta}}=0
    \end{equation}
    and
    \begin{equation}\label{boundh2del}
        ||w_{\infty}(x,t)||_{H^{2+\delta}}\leq \epsilon'+ S_{j_{0}(t)}+1.
    \end{equation}
so in particular we know that, for any $t_{0}>0$ 
\begin{equation}\label{sj0}
    \text{sup}_{t\in[t_{0},1]}||w_{\infty}(x,t)||_{H^{2+\delta}}\leq \epsilon'+ S_{j_{0}(t_{0})}+1
\end{equation}
and since $w_{\infty}(x,t)$ is a solution to \eqref{SQG}, using  Theorem 3.1 from \cite{ConstantinWu} tells us that $w_{\infty}$ is in $C^{\infty}$ for any $t>0$.

Finally, we have that, for $t=1$

$$||w_{\infty}(x,t)||_{H^{2+\delta}}\leq \sum_{j=1}^{\infty}||w_{c_{j}}(x,t)||_{H^{2+\delta}}+||\tilde{w}_{J}(x,t)-\tilde{w}_{J-1}(x,t)-T_{R_{J}}(w_{c_{J}}(x,t)||_{H^{2+\delta}}\leq 2^{-c_{1}+1}+\epsilon'$$
so, we can make the $H^{2+\delta}$ as small as we want by taking $c_{1}$ big and $\epsilon'$ small, and in particular $w_{\infty}(x,t)$ will be a global smooth solution to \eqref{SQG}.

    Step 3: We will now show that we have uniqueness, i.e. that, for any $M,\epsilon>0$, if we assume there exists a solution $\tilde{w}(x,t)$ to \eqref{SQG} fulfilling
    $$\text{sup}_{t\in[0,\epsilon]}||\tilde{w}(x,t)||_{H^{1}}\leq M$$
    then, for $t\in[0,\epsilon]$ $\tilde{w}(x,t)=w_{\infty}(x,t)$.

    For this, we note that, if we define $W(x,t):=\tilde{w}(x,t)-w_{\infty}(x,t)$ then

    $$\frac{\partial}{\partial t}W(x,t)+v(W)\cdot\nabla(W+w_{\infty})+v(w_{\infty})\cdot\nabla W+\Lambda^{\alpha} W=0$$
    so in particular
    $$\frac{\partial}{\partial t}||W||^2_{L^2}\leq -2\int_{\mathds{R}^2}W(v(W)\cdot\nabla w_{\infty}(x,t))dx$$
but then we have
\begin{equation}\label{c1bounded}
    |\int_{\mathds{R}^2}W(v(W)\cdot\nabla w_{\infty}(x,t))dx|\leq ||w_{\infty}(x,t)||_{C^1}||W||^2_{L^2}
\end{equation}
$$|\int_{\mathds{R}^2}W(v(W)\cdot\nabla w_{\infty}(x,t))dx|\leq ||w_{\infty}(x,t)||_{L^{\infty}}||W||_{L^2}||W||_{H^1},$$
so, for any $t_{0}\in[0,t]$ we have

    $$||W(x,t)||_{L^{2}}\leq Ct_{0}||W||_{H^1}e^{\int_{t_{0}}^{t}||w_{\infty}(x,s)||_{C^1}ds}$$
but, for $t\leq 1$

    \begin{align*}
      &||w_{\infty}(x,t)||_{C^1}\leq ||w_{\infty}(x,t)||_{H^{2+\delta}}\leq \epsilon'+ S_{j_{0}(t_{0})}+1
    \end{align*}
where we used \eqref{sj0}. Now, we note that for $t=\frac{1}{c_{j}}$ we already know that

  $$||W(x,t)||_{L^{2}}\leq \frac{C}{c_{j}}||W||_{H^1}e^{\epsilon'+ S_{j_{0}(\frac{C}{c_{j}})}+1}\leq\frac{1}{c_{j}}(M+1)e^{\epsilon'+ S_{j_{0}(\frac{1}{c_{j}})}+1}$$  
and, by definition of $j_{0}$ (see \eqref{j0}) we have that $j_{0}(\frac{1}{c_{j}})\leq j-1$, so
    $$||W(x,t)||_{L^{2}}\leq \frac{1}{c_{j}}(M+1)e^{\epsilon'+ S_{j-1}+1}\leq \frac{C(M+1)}{j}$$ 
    which tends to zero as $j$ tends to infinity so, for $t\in[0,1]$ $||W(x,t)||_{L^2}=0$ and therefore $\tilde{w}(x,t)=w_{\infty}(x,t)$. For $t>1$, we just use that $\text{sup}_{t\geq 1}||w_{\infty}(x,t)||_{C^1}\leq C$, and therefore \eqref{c1bounded} gives uniqueness. 
    Step 4: To end the proof, we need to show loss of regularity, and more precisely that there is a sequence of times $t_{n}$ such that
    $$\text{lim}_{n\rightarrow \infty}||w_{\infty}(x,t_{n})||_{H^{\beta}}=\infty.$$
    
    But we chose our $c_{j}$ so that

    $$2^{c_{j}}\geq 2 S_{j-1},$$
    and if $t_{c_{j}}$ is the time given by theorem \ref{wn} such that
        $$||w_{c_{j}}(x,t_{c_{j}})||_{H^{\beta}}\geq 2^{c_{j}}$$
        then $\frac{1}{c_{j+1}}\leq t_{c_{j}}$.

        Therefore, we have that
\begin{align*}
    &||w_{\infty}(x,t_{c_{j}})||_{H^{\beta}}\geq ||T_{R_{j}}(w_{c_{j}}(x,t_{c_{j}}))||_{H^{\beta}}-\sum_{i\in\mathds{N},i\neq j}||T_{R_{i}}(w_{c_{i}}(x,t_{c_{j}}))||\\
    &-\sum_{i=0}^{\infty}||\tilde{w}_{j}(x,t)-\tilde{w}_{j-1}(x,t)-T_{R_{j}}(w_{c_{j}}(x,t))||_{H^{\beta}}\geq 2^{c_{j}}-\sum_{i\in\mathds{N},i\neq j}||T_{R_{i}}(w_{c_{i}}(x,t_{c_{j}}))||-\epsilon' 
\end{align*}
      But
      $$\sum_{i=1}^{j-1}||T_{R_{i}}(w_{c_{i}}(x,t_{c_{j}}))||_{H^{\beta}}\leq 2\sum_{i=1}^{j-1}||T_{R_{i}}(w_{c_{i}}(x,t_{c_{j}}))||_{H^{6}}\leq S_{j-1}\leq 2^{c_{j}-2}$$
      $$\sum_{i=j+1}^{\infty}||T_{R_{i}}(w_{c_{i}}(x,t_{c_{j}}))||_{H^{\beta}}\leq \sum_{i=j+1}^{\infty}2^{-c_{i}}\leq 1$$
      so
      $$||w_{\infty}(x,t_{c_{j}})||_{H^{\beta}}\geq 2^{c_{j}-1}-1-\epsilon'$$
      and we are done.
\end{proof}

\section*{Acknowledgements}
This work is supported in part by the Spanish Ministry of Science
and Innovation, through the “Severo Ochoa Programme for Centres of Excellence in R$\&$D (CEX2019-000904-S)” and 114703GB-100. DC and LMZ were partially supported by the ERC Advanced Grant 788250. 

\bibliographystyle{alpha}

\end{document}